\documentclass[a4paper, 11pt]{article}
\title{Quotients and invariants of $\mathcal{AS}$-sets equipped with a finite group action}
\author{Fabien Priziac}
\date{}

\usepackage{amsfonts}
\usepackage{amsmath}
\usepackage{amsthm}
\usepackage{amssymb}
\usepackage[all]{xy}
\usepackage{graphicx}

\usepackage{hyperref}

\newtheorem{de}{Definition}[section]
\newtheorem{theo}[de]{Theorem}
\newtheorem{prop}[de]{Proposition}
\newtheorem{cor}[de]{Corollary}

\newtheorem{lem}[de]{Lemma}

\newcommand{\Hom}{\text{Hom}}

\theoremstyle{remark}
\newtheorem{rem}[de]{Remark}
\newtheorem{ex}[de]{Example}

\begin{document}

\maketitle

\begin{abstract}   
Using the geometric quotient of a real algebraic set by the action of a finite group $G$, we construct invariants of $G$-$\mathcal{AS}$-sets with respect to equivariant homeomorphisms with $\mathcal{AS}$-graph, including additive invariants with values in $\mathbb{Z}$.
\end{abstract}


\tableofcontents

\section{Introduction}

In \cite{MCP-VB}, C. McCrory and A. Parusi\'nski showed the existence (and the uniqueness) of an application $\beta$ which associates to any affine algebraic variety a polynomial in $\mathbb{Z}[u]$, and which is additive, invariant with respect to biregular isomorphisms and coincides with the classical Poincar\'e polynomial with coefficients in $\mathbb{Z}_2 := \mathbb{Z}/2\mathbb{Z}$ on compact nonsingular varieties : they called $\beta$ the virtual Poincar\'e polynomial. The coefficients of $\beta$ are called the virtual Betti numbers. In \cite{GF-MI}, G. Fichou extended $\beta$ to the wider and more flexible class of $\mathcal{AS}$-sets (see below definition \ref{defas}) and proved that the virtual Poincar\'e polynomial is an invariant with respect to Nash isomorphisms (i.e. semialgebraic and analytic isomorphisms). Finally, in \cite{MCP}, C. McCrory and A. Parusi\'nski associated to each $\mathcal{AS}$-set $X$ a filtered complex $\mathcal{N} C_*(X)$, called the Nash constructible filtration, which induces a spectral sequence from which one can recover the virtual Betti numbers. Since the Nash-constructible filtration is invariant under homeomorphisms with $\mathcal{AS}$-graph, so is the virtual Poincar\'e polynomial (it is actually invariant under bijections with $\mathcal{AS}$-graph).

In this paper, we consider $\mathcal{AS}$-sets equipped with a biregular action of a finite group $G$ (defined on the projective Zariski closure). We associate to each such $G$-$\mathcal{AS}$-set $X$ a filtered complex $\mathcal{N} C_*(X ; G)$, called the equivariant Nash constructible filtration (definition \ref{defequivnashcontr}). By construction, the equivariant Nash constructible filtration is invariant with respect to equivariant homeomorphisms with $\mathcal{AS}$-graph. As in the non-equivariant case, we can extract, from the induced spectral sequence, additive invariants $\beta_q(\, \cdot \, ; G)$, $q \in \mathbb{N}$, with coefficients in $\mathbb{Z}$ (theorem \ref{equivvbn}) : we call them the equivariant virtual Betti numbers (they are different from the equivariant virtual Betti numbers of \cite{GF}). They are invariant with respect to equivariant homeomorphisms with $\mathcal{AS}$-graph (because so is $\mathcal{N} C_*(\, \cdot \, ; G)$) and coincide with the dimensions of equivariant homology groups $H^G_*(X)$ if $X$ is compact and nonsingular. The equivariant homology $H^G_*(X)$ is the equivariant singular homology of $X$ with coefficients in $\mathbb{Z}_2$, which can be computed as the singular homology of the (topological) quotient of $X \times E_G$ by $G$, with $E_G$ being any contractible topological space equipped with a free action of $G$. The infinite real Stiefel manifold $V_{|G|}(\mathbb{R}^{\infty})$ is such a space.

The construction of the equivariant Nash contructible filtration deeply involves the properties of the geometric quotient of a real algebraic set by a finite group action. Indeed, if $X$ is a $G$-$\mathcal{AS}$-set, $\mathcal{N} C_*(X ; G)$ is an (inductive) limit of the Nash constructible filtrations of quotients $X \times V^k/G$, $k \in \mathbb{N}$. For $k \in \mathbb{N}$, $V^k$ is the real Stiefel manifold $V_{|G|}(\mathbb{R}^k)$ (it is a compact nonsingular real algebraic set) equipped with a free action of $G$. For the equivariant Nash constructible filtration to be well-defined, the quotients $X \times V^k/G$ have to be given a compatible $\mathcal{AS}$-structure. Actually, each $X \times V^k$ is an example of free $G$-$\mathcal{AS}$-set (definition \ref{deffreegasset}) and the quotient of a free $G$-$\mathcal{AS}$-set by $G$ can be given a well-defined $\mathcal{AS}$-structure (corollary \ref{quotientas}). 

For compact $G$-$\mathcal{AS}$-sets, the equivariant Nash constructible filtration induces a filtration on $H^G_*$ that we call the equivariant weight filtration (definition \ref{defhomequivweightfil}) in analogy with the non-equivariant case (\cite{MCP}). It is different from the equivariant weight filtration of \cite{Pri-EWF} (which was based on a different equivariant homology). 

Constructing the equivariant virtual Betti numbers, we have in mind their future use in the classification of real analytic germs. Denote by $\beta( \, \cdot \, ; G)$ the generating function of the equivariant virtual Betti numbers : we call it the equivariant virtual Poincar\'e series. Even if it is different from the equivariant virtual Poincar\'e series of \cite{GF}, it shares with it several properties including one (proposition \ref{propintmotequiv}) that should allow to define an invariant, in terms of (motivic) zeta functions, for an equivariant arc-analytic (see \cite{JB-AAE}) or equivariant blow-Nash equivalence (see \cite{GF-ZF}) of equivariant Nash germs, in a way similar to \cite{Pri-EBN} and \cite{Pri-ESNG}. An advantage of our equivariant virtual Poincar\'e series is that it is an invariant with respect to equivariant homeomorphism with $\mathcal{AS}$-graph, which we do not know for the equivariant virtual Poincar\'e series of \cite{GF}. On the other hand, G. Fichou's equivariant virtual Poincar\'e series encodes the dimension, which is not the case for ours : the two equivariant virtual Poincar\'e series have to be thought as complementary.
\\

The structure of the paper is as follows.

In section \ref{secquotalg}, we review the definition and properties of the geometric quotient of a real algebraic set by a finite group action. We focus on the properties that we need in the rest of the paper, such as functoriality or regularity, making precise some proofs.

In section \ref{sectionquotientAS}, we give a precise well-defined (up to Nash isomorphism) and functorial $\mathcal{AS}$-structure on the quotient of a free $G$-$\mathcal{AS}$-set, that is a $G$-$\mathcal{AS}$-set such that the action of $G$ on its compact arc-symmetric closure is free. This uses the key proposition \ref{quotientarcsym}.

In section \ref{secequivhomcohom}, we give the definition and the basic properties of the equivariant homology and cohomology with respect to which are constructed the equivariant $\mathcal{AS}$ invariants of section~\ref{sectequivnashconsfil}. The first paragraph is dedicated to the equivariant singular homology and cohomology, defined using the Borel construction. We show in particular that they can be computed using the Stiefel manifolds and the real geometric quotient (proposition \ref{equivsinghomindlim1} and corollary \ref{equivsinghomindlim2}). In the second part, we define the equivariant homology (and cohomology) with closed supports, as the homology of the inductive limit of the semialgebraic chain complexes with closed supports (\cite{MCP} Appendix) of the quotients $X \times V^k/G$, $k \in \mathbb{N}$. The two equivariant (co)homologies, singular and with closed supports, coincide on compact $G$-$\mathcal{AS}$-sets (lemma \ref{coincideequivhomsingclosed}).

In section \ref{sectequivnashconsfil}, we construct the equivariant Nash constructible filtration : if $X$ is $G$-$\mathcal{AS}$-set, it is the inductive limit of the Nash constructible filtration of the $\mathcal{AS}$ quotients $X \times V^k/G$. We then prove three essential properties of the equivariant Nash constructible filtration : it is functorial with respect to equivariant proper continuous maps with $\mathcal{AS}$ graph (theorem \ref{functequivnashfil}), it is additive with respect to equivariant closed inclusions in terms of a short exact sequence (theorem \ref{addequivnashfil}) and the lines of the induced spectral sequence are bounded (theorem \ref{equivspecseqbound}). We then define the equivariant virtual Betti numbers (theorem \ref{equivvbn}). At the end of this section, we also define the cohomological counterpart of the equivariant Nash constructible filtration (definition \ref{defdualequivnashfil}) and give its properties.

The final section of this paper is dedicated to some further properties of the equivariant virtual Poincar\'e series, including the statement that should be useful for the classification of real analytic germs (proposition \ref{propintmotequiv}). The other proposition \ref{equivvpoinctrivfree} gives the behaviour of the equivariant virtual Poincar\'e series on free $G$-$\mathcal{AS}$-sets and on $\mathcal{AS}$-sets equipped with a trivial action of $G$.

Throughout this paper, $G$ will always denote a finite group and $\mathbb{Z}_2 := \mathbb{Z}/2 \mathbb{Z}$ will denote the two-elements field. 
\\

{\bf Acknowledgements.} The author wishes to thank J.-B. Campesato and G. Fichou for useful discussions and comments. 

\section{Quotient of a real algebraic set by a finite group action} \label{secquotalg}

\subsection{Construction of the geometric quotient of a real algebraic set by a polynomial finite group action}

Let $G$ be a finite group of order $N \in \mathbb{N} \setminus \{0\}$.

In this section, we will consider a real algebraic set $X \subset \mathbb{R}^d$ on which $G$ acts via polynomial maps $\alpha_g : X \rightarrow X$, $g \in G$. We are going to recall the construction of the semialgebraic geometric quotient of $X$ by $G$ (see also for instance \cite{Pro} or \cite{Ozan}). 
\\

Denote by $\mathcal{P}(X) := \mathbb{R}[x_1, \ldots, x_d]/I(X)$ the $\mathbb{R}$-algebra of polynomial functions on $X$. The action of $G$ on $X$ induces an action of $G$ on $\mathcal{P}(X)$, defined by 
$$g \cdot f := f \circ \alpha_g^{-1}$$
if $f \in \mathcal{P}(X)$ and $g \in G$. Since $\mathcal{P}(X)$ is a finitely generated $\mathbb{R}$-algebra and $G$ is finite, the subalgebra $\mathcal{P}(X)^G$ of invariant polynomial functions on $X$ is a finitely generated $\mathbb{R}$-algebra as well (see for instance \cite{Sha} Algebraic Appendix, section 4) : there exist invariant polynomial functions $p_1, \ldots, p_m$ on $X$ which generate $\mathcal{P}(X)^G$ as an $\mathbb{R}$-algebra.

Now, consider the complexification $X_{\mathbb{C}} \subset \mathbb{C}^d$ of $X$ : it is, by definition, the Zariski closure of $X$ considered as a subset of $\mathbb{C}^d$ (see \cite{AK} II.2). Notice that the coordinate ring $\mathbb{C}[X_{\mathbb{C}}]$ of $X_{\mathbb{C}}$ is the tensor product $\mathcal{P}(X) \otimes_{\mathbb{R}} \mathbb{C}$. We extend linearly the action of $G$ on $\mathcal{P}(X)$ into an action on $\mathbb{C}[X_{\mathbb{C}}]$, which corresponds (via the contrafunctorial equivalence between the category of complex algebraic sets and the category of reduced finitely generated $\mathbb{C}$-algebras, given by the Nullstellensatz) to an action of $G$ on $X_{\mathbb{C}}$ (the ``complexified'' action being given by the same polynomials with real coefficients as for the action of $G$ on $X$).   

Since the action of $G$ on $\mathbb{C}[X_{\mathbb{C}}] = \mathcal{P}(X) \otimes_{\mathbb{R}} \mathbb{C}$ is induced by linear extension, we have $\mathbb{C}[X_{\mathbb{C}}]^G = \mathcal{P}(X)^G \otimes_{\mathbb{R}} \mathbb{C}$ and the $\mathbb{C}$-algebra $\mathbb{C}[X_{\mathbb{C}}]^G$ is then generated (as a $\mathbb{C}$-algebra) by the invariant real polynomial functions $p_1, \ldots, p_m$ (considered as functions on $X_{\mathbb{C}}$).

Therefore, the reduced finitely generated $\mathbb{C}$-algebra $\mathbb{C}[X_{\mathbb{C}}]^G$ corresponds to the complex algebraic subset $Z := V\left(\left\{P \in \mathbb{C}[z_1, \ldots, z_m]~|~P(p_1, \ldots, p_m) = 0\right\}\right)$ of $\mathbb{C}^m$ and the inclusion $\mathbb{C}[X_{\mathbb{C}}]^G \subset \mathbb{C}[X_{\mathbb{C}}]$ corresponds to the polynomial map  

$$\pi : \begin{array}{ccc} 
		X_{\mathbb{C}} & \rightarrow & Z \\
		x & \mapsto & (p_1(x), \ldots, p_m(x))
		\end{array}.$$
The morphism $\pi$ is a finite map (\cite{Sha} I.5.3 Example 1), hence surjective (\cite{Sha} I.5.3 Theorem 4). 

Finally, we shall consider the image $Y := \pi(X)$ of $X$ by $\pi$. Since $\pi$ is given by real polynomials, $Y$ is a semialgebraic subset of $\mathbb{R}^m$ (\cite{BCR} Chap. 2). Notice also that, if we denote by $W$ the Zariski closure of $Y$ in $\mathbb{R}^m$, then $Z = W_{\mathbb{C}}$ (see \cite{Ozan} Lemma 1.3).  

\begin{de}
We call $\pi : X \rightarrow Y$ the geometric quotient of $X$ by $G$. By abuse of terminology, we will also call $Y$ the geometric quotient of $X$ by $G$.
\end{de}

\begin{rem} \label{indepgen} \begin{itemize}
	\item A (direct) correspondence between the real algebraic set $W$ and the real finitely generated $\mathbb{R}$-algebra $\mathcal{P}(X)^G$ is given by the real Nullstellensatz (\cite{BCR} Theorem 4.1.4).
	\item Different sets of generators for $\mathcal{P}(X)^G$ provides isomorphic geometric quotients, via (real) polynomial mappings (consider the complexified algebra $\mathbb{C}[Z] = \mathcal{P}(X)^G \otimes_{\mathbb{R}} \mathbb{C}$ : two different sets of real invariant generators provide isomorphic algebraic sets via real polynomial mappings).
	\item If $x_1, x_2 \in X$ then $\pi(x_1) = \pi(x_2)$ if and only if there exists $g \in G$ such that $x_2 = \alpha_g(x_1)$ (see for instance \cite{Sha} Chapter 1, section 2.3, Example 11).
	\end{itemize}
\end{rem}

\begin{ex} \label{firstexquotient} \begin{enumerate}
	\item Consider the action of $G := \mathbb{Z}/2\mathbb{Z}$ on $X := \mathbb{R}^2$ given by the involution $\sigma : (x_1,x_2) \mapsto (-x_1,x_2)$. Then $\mathcal{P}(X) = \mathbb{R}[X_1,X_2]$, $\mathcal{P}(X)^G = \mathbb{R} \langle X_1^2, X_2 \rangle$ and $\pi : (x_1, x_2) \mapsto (x_1^2, x_2)$, so that the geometric quotient of $X$ by $G$ is the half-plane $\left\{(y_1, y_2) \in \mathbb{R}^2~|~y_1 \geq 0\right\}$.
	\item Now, consider the action of $G$ on $X$ given by $\sigma : (x_1,x_2) \mapsto (-x_1,-x_2)$. We have $\mathcal{P}(X)^G = \mathbb{R} \langle X_1^2, X_2^2, X_1 X_2 \rangle$ and $\pi : (x_1, x_2) \mapsto (x_1^2, x_2^2, x_1 x_2)$, so that the geometric quotient of $X$ by $G$ is the half elliptic cone $\left\{(y_1, y_2, z) \in \mathbb{R}^3~|~z^2 = y_1 y_2,~y_1 \geq 0,~ y_2 \geq 0\right\}$.
\end{enumerate}
\end{ex}

\subsection{Basic properties of the construction of the real geometric quotient}

We give some basic properties of the previous construction. First, it is functorial with respect to polynomial maps :

\begin{lem} \label{indepemb} Let $X$ and $X'$ be real algebraic sets on which $G$ acts via polynomial maps, and let $\pi : X \rightarrow Y$ and $\pi' : X' \rightarrow Y'$ be the respective geometric quotients of $X$ and $X'$ by $G$. If $\psi : X \rightarrow X'$ is an equivariant polynomial map, there exists a unique polynomial map $\rho : Y \rightarrow Y'$ such that the following diagram 
$$\begin{array}{ccc} 
X & \stackrel{\psi}{\longrightarrow} & X'\\
 ~ \downarrow {\scriptstyle \pi} & & ~ \downarrow {\scriptstyle \pi'}\\
Y & \stackrel{\rho}{\longrightarrow} & Y'
\end{array}$$  
commutes.
\end{lem}

\begin{proof}
The polynomial map $\psi : X \rightarrow X'$ induces, by complexification, a polynomial map $\psi_{\mathbb{C}} : X_{\mathbb{C}} \rightarrow X'_{\mathbb{C}}$ (given by the same polynomials with real coefficients), which is also equivariant (with respect to the complexified actions of $G$) thanks to the functoriality of the complexification process.

This morphism corresponds to a morphism of $\mathbb{C}$-algebras $\psi_{\mathbb{C}}^* : \mathbb{C}[X'_{\mathbb{C}}] \rightarrow \mathbb{C}[X_{\mathbb{C}}]$. We then consider the restriction $\mathbb{C}[X'_{\mathbb{C}}]^G \rightarrow \mathbb{C}[X_{\mathbb{C}}]^G$ of this last morphism ($\psi_{\mathbb{C}}$ is equivariant), which corresponds to a polynomial map $\rho : Z \rightarrow Z'$, (where $Z$ and $Z'$ are the respective geometric quotient of $X_{\mathbb{C}}$ and $X'_{\mathbb{C}}$), given by real polynomials.

Precisely, we can describe the polynomial map $\rho$ in the following way. Suppose that $(\mathcal{P}(X))^G$ is generated by the real polynomial functions $p_1, \ldots, p_m$ and that $(\mathcal{P}(X'))^G$ is generated by $q_1, \ldots, q_{m'}$. For each $j \in \{1, \ldots, m'\}$, we have $\psi_{\mathbb{C}}^*(q_j) = q_j \circ \psi_{\mathbb{C}} = q_j \circ \psi \in (\mathcal{P}(X))^G$ ($\psi$ is equivariant), so that $\psi^*_{\mathbb{C}}(q_j)$ is a real polynomial $Q_j$ in the real polynomial functions $p_1, \ldots, p_m$. If $z \in Z$, we then have
\begin{equation} \label{quotientofamap}
\rho(z) = (Q_1(z), \ldots, Q_{m'})
\end{equation}

Using this description of the map $\rho$, we can check that the diagram 
$$\begin{array}{ccc} 
X_{\mathbb{C}} & \stackrel{\psi_{\mathbb{C}}}{\longrightarrow} & X'_{\mathbb{C}}\\
 ~ \downarrow {\scriptstyle \pi} & & ~ \downarrow {\scriptstyle \pi'}\\
Z & \stackrel{\rho}{\longrightarrow} & Z'
\end{array}$$  
commutes. In particular, $\rho(Y) = \rho(\pi(X)) = \pi'(\psi(X)) \subset Y'$ and the diagram 
$$\begin{array}{ccc} 
X & \stackrel{\psi}{\longrightarrow} & X'\\
 ~ \downarrow {\scriptstyle \pi} & & ~ \downarrow {\scriptstyle \pi'}\\
Y & \stackrel{\rho}{\longrightarrow} & Y'
\end{array}$$  
commutes as well (notice that we also have a restriction $\rho : W \rightarrow W'$ if $W$ and $W'$ denote the respective (real) Zariski closures of $Y$ and $Y'$). 

Furthermore, we can also check functoriality by using the description (\ref{quotientofamap}) of the map $\rho$.
\end{proof}

\begin{rem} In particular, if $X$ is embedded in some other $\mathbb{R}^D$ via an equivariant polynomial embedding, the semialgebraic geometric quotient of $X$ and the semialgebraic geometric quotient of its embedding are isomorphic via polynomial maps.
\end{rem}

If we have an equivariant inclusion of real algebraic sets $X \subset X'$, the geometric quotient of $X$ can be naturally embedded in the geometric quotient of $X'$ :

\begin{lem} \label{quotientinabigger} Keep the notation of previous lemma \ref{indepemb} and suppose that we have an equivariant inclusion $X \subset X'$. It induces an equivariant inclusion $Y \subset Y'$ of the corresponding geometric quotients.
\end{lem}

\begin{proof} The inclusion $i : X \hookrightarrow X'$ induces a surjective morphism of $\mathbb{R}$-algebras $i^* : \mathcal{P}(X') \rightarrow \mathcal{P}(X)$ given by the restriction. Since $i$ is equivariant, we can consider the restriction morphism $i^* : \mathcal{P}(X')^G \rightarrow \mathcal{P}(X)^G$, which is surjective as well since $G$ is finite (if $f = i^*(h) \in \mathcal{P}(X)^G$ with $h \in \mathcal{P}(X')$, write $f = i^* \left(\frac{1}{|G|} \sum_{g \in G} g \cdot h\right)$).

Consequently, if the $\mathbb{R}$-algebra $\mathcal{P}(X')^G$ is generated by invariant polynomial functons $p_1, \ldots, p_m$, the invariant polynomial functions $i^*(p_1), \ldots, i^*(p_m)$ generate the $\mathbb{R}$-algebra $\mathcal{P}(X)^G$.

It follows that the morphism $i_{\mathbb{C}}^* : \mathbb{C}[Z'] = \mathbb{C}[X'_{\mathbb{C}}]^G \rightarrow \mathbb{C}[Z] = \mathbb{C}[X_{\mathbb{C}}]^G$ (which is also surjective) corresponds to the closed embedding $Z \subset \mathbb{C}^m \rightarrow Z' \subset \mathbb{C}^m$ given by $(z_1, \ldots, z_m) \mapsto (z_1, \ldots, z_m)$, so we can consider it as an inclusion $Z \subset Z' \subset \mathbb{C}^m$. Hence the inclusion $Y \subset Y' \subset \mathbb{R}^m$.

\end{proof}

\begin{rem} In particular, if the action of $G$ on $X$ comes from the action of $G$ on an algebraic set $X' \supset X$, its geometric quotient is (up to polynomial isomorphim) the image of $X$ under the quotient map of $X'$.
\end{rem}

\begin{ex} \label{secexquotient} \begin{enumerate}
	\item Consider the action of $G := \mathbb{Z}/2\mathbb{Z}$ on the unit circle $\mathbb{S}^1$ given by the involution $\sigma : (x_1,x_2) \mapsto (-x_1,x_2)$. We then apply the quotient map $\pi$ of example \ref{firstexquotient} (i) to obtain that the geometric quotient of $\mathbb{S}^1$ by $G$ is the semialgebraic subset $\left\{(y_1, y_2) \in \mathbb{R}^2~|~y_1 + y^2 = 1,~y_1 \geq 0\right\}$. 
	\item Now, consider the free action of $G$ on $\mathbb{S}^1$ given by the involution $\sigma : (x_1,x_2) \mapsto (-x_1,-x_2)$.  We apply the quotient map $\pi$ of example \ref{firstexquotient} (ii) and the geometric quotient of $\mathbb{S}^1$ by $G$ is the section of the half elliptic cone by the hyperplane of equation $y_1 + y_2 = 1$, hence an ellipse (it is then polynomially isomorphic to the unit circle).
	\item Consider the free action of $G$ on the hyperbola $X := \{x_1 x_2 = 1\}$ of $\mathbb{R}^2$ given by the same involution $\sigma : (x_1,x_2) \mapsto (-x_1,-x_2)$. The geometric quotient of $X$ by $G$ is the section of the half elliptic cone by the hyperplane $y_3 = 1$, hence the ``half-hyperbola'' $\{y_1 y_2 = 1, y_1, y_2 \geq 0, y_3 = 1\}$ (polynomially isomorphic to the half-hyperbola $\{x_1 x_2 = 1, x_1, x_2 \geq 0\}$ of $\mathbb{R}^2$).
\end{enumerate}
\end{ex}

\subsection{Geometric properties of the real geometric quotient}

Let $X$ be a real algebraic set acted by $G$ via polynomial maps $\alpha_g : X \rightarrow X$, $g \in G$, and $\pi : X \rightarrow Y \subset W$ be the associated quotient map, where $W$ is the Zariski closure of $Y$.

We will give some geometric properties of the quotient : we begin by checking that the geometric quotient preserves the dimension. 

\begin{lem} \label{quotientdimension} We have
$$\dim Y = \dim W = \dim X.$$
\end{lem}

\begin{proof}
The left-hand equality follows from the fact that $W$ is the Zariski closure of $Y$ (\cite{BCR} Proposition 2.8.2). To establish the right-hand side equality, we use the fact that the dimension of $X$ as a real algebraic set is equal to the dimension of $X_{\mathbb{C}}$ as a complex algebraic set (see \cite{AK} II.2, in particular Proposition 2.2.1.d and Proposition 2.2.5.b). Furthermore, the geometric quotient map $\pi : X_{\mathbb{C}} \rightarrow Z$ is a finite map (see \cite{Sha} I.5.3 Example 1) : this means that the coordinate ring of $X_{\mathbb{C}}$ is integral over the coordinate ring of $Z$, in particular the two coordinate rings have the same Krull dimension, that is $\dim X_{\mathbb{C}} = \dim Z$. Since $Z = W_{\mathbb{C}}$, we obtain $\dim W = \dim X$.
\end{proof}

We then prove the essential fact that the image under the quotient map of a nonsingular point with trivial stabilizer is nonsingular :

\begin{prop} \label{quotientnonsingiffree} Let $x$ be a nonsingular point of $X$. If $G_x = \{e\}$, then $\pi(x)$ is a nonsingular point of $W$. 
\end{prop}

\begin{proof} We first follow the proof of \cite{Sha} II.2.1 Example, showing that it works over the real algebraic sets as well. 
\\

Denote by $n$ the dimension of the real algebraic set $X$. Since $x$ is a nonsingular point of $X$, by definition, the local ring $\mathcal{R}_{X,x}$ at $x$ is a regular local ring of dimension $n$. In particular, $\dim_{\mathbb{R}} m_x / m_x^2 = n$, where $m_x$ denotes the maximal ideal of the germs of functions of $\mathcal{R}_{X,x}$ which vanish at $x$ (see \cite{BCR} section 3.3).

By Proposition 3.3.7 of \cite{BCR}, there exist functions $u_1, \ldots, u_n$ in $\mathcal{R}_{X,x}$ generating $m_x$ and such that $\overline{u_1}, \ldots, \overline{u_n}$ is a basis of $m_x / m_x^2$ as a $\mathbb{R}$-vector space (such a family $u_1, \ldots, u_n$ is called a regular system of parameters of $\mathcal{R}_{X,x}$). We can assume that $u_1, \ldots, u_n$ are given by polynomial functions on $X$. 

In the first part of the proof, we are going to construct, from the $u_i$'s, generators $w_1, \ldots, w_n$ of the maximal ideal $m_{\pi(x)}$ of the local ring $\mathcal{R}_{W, \pi(x)}$, whose classes in $m_{\pi(x)}/m_{\pi(x)}^2$ are linearly independent. In particular, the dimension of the dual $m_{\pi(x)}/m_{\pi(x)}^2$ of the Zariski tangent space at $\pi(x)$ is equal to $n$. In a second part, we will prove that the ring $\mathcal{R}_{W, \pi(x)}$ is also $n$-dimensional, so that it is a regular local ring of dimension $n$ i.e. $\pi(x)$ is a nonsingular point of $W$ (the dimension of $W$ is $n$ by previous lemma \ref{quotientdimension}).
\\

The first step will be to construct a regular system of parameters for $\mathcal{R}_{X,x}$ with elements in $\mathcal{P}(X)^G$. We begin by showing that we can assume the $u_i$'s to belong to $m_{\alpha_g(x)}^2$ for all $g \in G$ different from $e$. Indeed, for each $g \in G$ different from $e$, consider a polynomial function $h_g$ on $X$ which verifies $h_g(x) = 1$ and $h_g(\alpha_g(x)) = 0$ : this is possible because the ideal of polynomial functions vanishing at $x$ cannot be equal to the ideal of polynomial functions vanishing at $\alpha_g(x)$ (otherwise, $\alpha_g(x)$ would be equal to $x$). Then denote by $h$ the product of all the polynomial functions $h_g$~: we have $h(x) = 1$ and, for all $g \in G$ different from $e$, $h(\alpha_g(x)) = 0$. Finally, multiply each $u_i$ by the square of $h$ : if we denote $h_0 := h -1 \in m_x$, we obtain 
$$u_i h^2 = u_i + 2 u_i h_0 + {h_0}^2 \equiv u_i \! \mod m_x^2,$$
while $u_i h^2 \in m_{\alpha_g(x)}^2$ if $g \neq e$.

Now, consider the invariant polynomial functions 
$$v_i := \frac{1}{|G|} \sum_{g \in G} u_i \circ \alpha_g$$
of $\mathcal{P}(X)^G$. For each $i$, we have $v_i \in m_x$ (because $u_i \in m_{g(x)}$ for all $g \in G$) and $v_i \equiv \frac{1}{|G|}u_i \! \mod m_x^2$ (because, for $g \neq e$, $u_i \in m_{\alpha_g(x)}^2$ i.e $u_i \circ \alpha_g \in m_x^2$) so that the $v_i$'s form a regular system of parameters of $\mathcal{R}_{X,x}$ as well. 

Let $i \in \{1, \ldots, n\}$. Since the $\mathbb{R}$-algebra $\mathcal{P}(X)^G$ is generated by polynomial functions $p_1, \ldots, p_m$ and since, for all $x \in X$, $\pi(x) = (p_1(x), \ldots, p_m(x))$, there exist a polynomial $w_i \in \mathbb{R}[X_1, \ldots, X_m]$ such that $v_i = w_i \circ \pi$. Because $\pi$ has values in $W$, we are going to consider $w_i$ as a polynomial function of $\mathcal{P}(W)$ and, furthermore, as an element of the maximal ideal $m_{\pi(x)}$ of $\mathcal{R}_{W,\pi(x)}$ (since $w_i(\pi(x)) = v_i(x) = 0$). 

Finally, we prove that the $w_i$'s generate the ideal $m_{\pi(x)}$. Let $h \in m_{\pi(x)} \cap \mathcal{P}(W)$. Then $h \circ \pi \in m_x$ so that there exist $h_1, \ldots, h_n \in \mathcal{R}_{X,x}$ such that $h \circ \pi = \sum_{i = 1}^n h_i v_i$. Moreover, $h \circ \pi \in \mathcal{P}(X)^G$ (because $\pi \circ \alpha_g = \pi$) and consequently
$$h \circ \pi = \frac{1}{|G|} \sum_{g \in G} h \circ \pi \circ \alpha_g = \frac{1}{|G|} \sum_{g \in G} \sum_{i=1}^n (h_i \circ \alpha_g) (v_i \circ \alpha_g) = \sum_{i=1}^n v_i \left(\frac{1}{|G|} \sum_{g \in G} h_i \circ \alpha_g\right)$$
(recall that the $v_i$'s are in $\mathcal{P}(X)^G$). Now, for each $i = 1, \ldots, n$, $\frac{1}{|G|} \sum_{g \in G} h_i \circ \alpha_g \in \mathcal{P}(X)^G$ so there exists $q_i \in \mathcal{P}(W)$ such that $\frac{1}{|G|} \sum_{g \in G} h_i \circ \alpha_g = q_i \circ \pi$. Finally, we obtain
$$h \circ \pi = \sum_{i=1}^n (q_i \circ \pi) (w_i \circ \pi)$$
i.e. $h - \sum_{i=1}^n q_i w_i \in \ker \pi^*$ (recall that $Z := W_{\mathbb{C}} = V(\ker \pi^*)$ and that $\ker \pi^*$ is a radical ideal). Since the polynomial function $h - \sum_{i=1}^n q_i w_i$ has real coefficients and since $W \subset Z$, we have $h - \sum_{i=1}^n q_i w_i \in I(W)$ and $h = \sum_{i=1}^n q_i w_i$ in $\mathcal{P}(W)$.

We then show that the classes $\overline{w_1}, \ldots, \overline{w_n}$ in $m_{\pi(x)}/m_{\pi(x)}^2$ are linearly independent over $\mathbb{R}$~: let $\lambda_1, \ldots, \lambda_n \in \mathbb{R}$ such that $\lambda_1 \overline{w_1} + \ldots + \lambda_n \overline{w_n} = 0$, then compose with $\pi$ to obtain
$$\lambda_1 \overline{w_1 \circ \pi} + \ldots + \lambda_n \overline{w_n \circ \pi} = \lambda_1 \overline{v_1} + \ldots + \lambda_n \overline{v_n} = 0,$$
and use the fact that the $v_i$'s form a regular system of parameters of $\mathcal{R}_{X,x}$.

As a conclusion, the $\overline{w_i}$'s form a basis of the $\mathbb{R}$-vector space $m_{\pi(x)}/m_{\pi(x)}^2$, which is then $n$-dimensional. In order to show that $\pi(x)$ is a nonsingular point of $W$, it remains to show that the dimension of the ring $\mathcal{R}_{W, \pi(x)}$ is $n$ as well.
\\

Recall that $x$ is a nonsingular point of $X$. In particular, there exists a unique irreducible $n$-dimensional component $V$ of $X$ such that $x$ is a nonsingular point of $V$ (\cite{BCR} Proposition 3.3.10). Moreover, if $X = \bigcup_j V^j$ is the decomposition of $X$ into (real) algebraic irreducible components, $X_{\mathbb{C}} = \bigcup_j V^j_{\mathbb{C}}$ is the decomposition of $X_{\mathbb{C}}$ into (complex) algebraic irreducible components (\cite{AK} II.2). Therefore, $x$ belongs to a unique irreducible component of $X_{\mathbb{C}}$, namely $V_{\mathbb{C}}$, which is also $n$-dimensional.

Now, the group $G$ acts on the set of irreducible components of $X_{\mathbb{C}}$, so we can write $X_{\mathbb{C}}$ as the union of $G$-stable algebraic subsets $S^k_{\mathbb{C}}$, where each $S^k_{\mathbb{C}}$ is the union of the irreducible components of $X_{\mathbb{C}}$ being in a same orbit. Then $Z = \bigcup_k \pi\left(S^k_{\mathbb{C}}\right)$ is the decomposition of $Z$ into irreducible components (indeed, each $\pi\left(S^k_{\mathbb{C}}\right)$ is algebraic -- see lemma \ref{quotientinabigger} -- and irreducible, since it is the image by $\pi$ of an irreducible component of $X_{\mathbb{C}}$) and we obtain the decomposition 
$$W = \bigcup_k \left[\pi\left(S^k_{\mathbb{C}}\right) \cap \mathbb{R}^m\right]$$
for $W$ into (real algebraic) irreducible components.

The point $\pi(x)$ then belongs to a unique irreducible component of $W$, namely $W_0 := \pi(V_{\mathbb{C}}) \cap \mathbb{R}^m$, which is $n$-dimensional by lemma \ref{quotientdimension}. As a consequence, $\mathcal{R}_{W, \pi(x)} = \mathcal{R}_{W_0, \pi(x)}$ and, since $W_0$ is irreducible, 
$$\dim \mathcal{R}_{W_0, \pi(x)} = \dim W_0 = n,$$
which the dimension of $m_{\pi(x)}/m_{\pi(x)}^2$, so that $\mathcal{R}_{W, \pi(x)} $ is a regular local ring of dimension $n$, i.e. $\pi(x)$ is a nonsingular point of $W$. 
\end{proof}

In the proof, we moreover showed that $\pi$ induces an isomorphism 
$$m_{\pi(x)}/m_{\pi(x)}^2 \rightarrow m_{x}/m_{x}^2$$
between the duals of the respective Zariski tangent spaces of $W$ at $\pi(x)$ and of $X$ at $x$. As a consequence :

\begin{prop} \label{localdiff} Suppose that $x$ is a nonsingular point of $X$ and that $G_{x} = \{e\}$. There exist a semialgebraic open neighborhood $U$ of $x$ in $X$ and a semialgebraic open neighborhood $U'$ of $\pi(x)$ in $W$ such that $\pi_{| U}$ is a Nash (i.e. semialgebraic and analytic) diffeomorphism from $U$ to~$U'$.
\end{prop}

\begin{proof} In previous proposition \ref{quotientnonsingiffree}, we proved that $\pi(x)$ is a nonsingular point of $W$ and that $\pi$ induces an isomorphism
$$T_x X \rightarrow T_{\pi(x)} W.$$
We then use Proposition 8.1.2 of \cite{BCR} to conclude.
\end{proof}

\subsection{Real geometric quotient of a product}

At some moment in the next part of this paper, we will need to consider the geometric quotient of the cartesian product of real algebraic sets under the action of the product group. Precisely, we have the following property :

\begin{lem} \label{quotientprod} Let $H$ be another finite group. Let $X \subset \mathbb{R}^{d}$ and $X' \subset \mathbb{R}^{d'}$ be two algebraic sets equipped with polynomial actions of $G$ and $H$ respectively, and let $\pi : X \rightarrow Y \subset \mathbb{R}^m$ and $\pi' : X' \rightarrow Y' \subset \mathbb{R}^{m'}$ be the corresponding quotient maps. 

The quotient map corresponding to the induced action of $G \times H$ on $X \times X'$ is
$$\pi \times \pi' : \begin{array}{ccc} 
		X \times X'  & \rightarrow & Y \times Y' \\
		(x,x') & \mapsto & (\pi(x), \pi'(x'))
		\end{array}$$
\end{lem}

\begin{proof} We have $\mathcal{P}(X) \otimes_{\mathbb{R}} \mathcal{P}(X') \cong \mathcal{P}(X \times X')$ via the equivariant (with respect to the induced actions of $G \times H$) isomorphism $\overline{f} \otimes \overline{h} \mapsto \overline{f \times h}$ (notice that $\mathbb{C}[X_{\mathbb{C}}] \otimes \mathbb{C}[X'_{\mathbb{C}}] \cong \mathbb{C}[X_{\mathbb{C}} \times X'_{\mathbb{C}}]$ via the same equivariant isomorphism).

Furthermore, $\left(\mathcal{P}(X) \otimes_{\mathbb{R}} \mathcal{P}(X')\right)^{G \times H} = \mathcal{P}(X)^G \otimes_{\mathbb{R}} \mathcal{P}(X')^H$ so that, if $p_1, \ldots, p_{m}$ are generators of $\mathcal{P}(X)^G$ and $q_1, \ldots, q_{m'}$ are generators of $\mathcal{P}(X')^H$, $p_1 \otimes 1, \ldots, p_{m} \otimes 1, 1 \otimes q_1, \ldots, 1 \otimes q_{m'}$ are generators of $\left(\mathcal{P}(X) \otimes_{\mathbb{R}} \mathcal{P}(X')\right)^{G \times H}$ and 
$$\begin{array}{ccc} 
		X \times X'  & \rightarrow & Y \times Y' \\
		(x,x') & \mapsto & (p_1(x) , \ldots, p_{m}(x) , q_1(x'), \ldots, q_{m'}(x'))
		\end{array}$$
is the quotient map.
\end{proof}

\section{Quotient of an arc-symmetric set by a free finite group action} \label{sectionquotientAS}

\subsection{Quotient of a semialgebraic set by a polynomial finite group action}

Let $X \subset \mathbb{R}^d$ be a semialgebraic set. We suppose that the Zariski closure $\overline{X}^{\mathcal{Z}}$ of $X$ is equipped with a polynomial action of $G$ which globally stabilizes $X$ ($G$ still denotes a finite group).

Let $\pi : \overline{X}^{\mathcal{Z}} \rightarrow Y$ be the geometric quotient of $\overline{X}^{\mathcal{Z}}$ by $G$.

\begin{de} We call the restriction $\pi : X \rightarrow \pi(X)$, or simply $\pi(X)$, the geometric quotient of $X$ by $G$.
\end{de}

\begin{rem} As $\pi\left(\overline{X}^{\mathcal{Z}}\right)$, the geometric quotient of $X$ by $G$ is well-defined up to polynomial isomorphism. Moreover, if the action of $G$ on $\overline{X}^{\mathcal{Z}}$ comes from a polynomial action on a real algebraic set $X'$ containing $\overline{X}^{\mathcal{Z}}$, the geometric quotient of $X$ can be obtained as the image of $X$ under the quotient map of $X'$ : see remark \ref{indepgen}, lemma \ref{indepemb} and lemma \ref{quotientinabigger}. 
\end{rem}

\begin{ex}  Consider the action of the involution $\sigma : (x_1, x_2) \mapsto (-x_1,x_2)$ on the upper half-plane $\{x_2 \geq 0\}$ of $\mathbb{R}^2$. Its geometric quotient is the first quadrant $\{y_1 \geq 0, \, y_2 \geq 0\}$ of $\mathbb{R}^2$ (see example \ref{firstexquotient}).
\end{ex}

\begin{lem} \label{quotientproper} The continuous map $\pi : X \rightarrow \pi(X)$ is proper, closed and open. Furthermore, it is homeomorphic to the topological quotient map $q : X \rightarrow X/G$.
\end{lem}

In order to prove this lemma, we will suppose without loss of generality that $X$ is globally stabilized under an orthogonal linear action :

\begin{lem} \label{linearisationfinigroup} Let $V \subset \mathbb{R}^d$ be a real algebraic set on which $G$ acts via polynomial isomorphisms. There exists an equivariant polynomial isomorphism $\varphi : V \rightarrow V'$, where $V' \subset \mathbb{R}^D$ is a real algebraic set equipped with a linear action of $G$ on $\mathbb{R}^D$ given by permutation matrices.
\end{lem}  

\begin{proof}

Indeed, denote $G := \{g_1, \ldots, g_N\}$ where $g_1$ is the identity element of $G$ and consider the morphism

$$\varphi : \begin{array}{ccc} 
		V & \rightarrow & V \times \cdots \times V \\
		x & \mapsto & (x, \alpha_{g_2^{-1}}(x), \ldots, \alpha_{g_N^{-1}}(x))
		\end{array}$$

Now, equip the cartesian product $V \times \cdots \times V$ with the action of $G$ given by the permutations induced by the product in $G$. The morphism $\varphi$ is then equivariant. 

Furthermore, $\varphi$ induces a polynomial isomorphism between $V$ and its image $\varphi(V)$ :  the direct image of $V$ by $\varphi$ is an algebraic subset of $\mathbb{R}^d \times \cdots \times \mathbb{R}^d$ given by the equations $y_1 \in V$, $y_i = \alpha_{g_i}(y_1)$ (recall that each $\alpha_i$ is a polynomial isomorphism).
\end{proof}
 
\begin{proof}[Proof of lemma \ref{quotientproper}] Up to polynomial isomorphism, we can then suppose that $X$ is globally stabilized by a linear orthogonal action of $G$ on $\mathbb{R}^d$, and $\pi : X \rightarrow \pi(X)$ is then the restriction of the corresponding quotient map $\pi = (p_1, \ldots, p_m) : \mathbb{R}^d \rightarrow \mathbb{R}^m$ on $\mathbb{R}^d$ (lemma \ref{indepemb} and lemma \ref{quotientinabigger}). 

The map $\pi : \mathbb{R}^d \rightarrow \mathbb{R}^m$ is proper. Indeed, let us recall the argument of \cite{Sch75} : the map
$$\phi : \begin{array}{ccc}
\mathbb{R}^d & \rightarrow & \mathbb{R} \\
x = (x_1, \ldots, x_d) & \mapsto & \langle x,x \rangle  = x_1^2+ \ldots + x_d^2 
\end{array}$$
is a proper polynomial map, which is invariant under right composition with the orthogonal action of $G$ on $\mathbb{R}^d$. Therefore, $\phi \in \mathcal{P}(\mathbb{R}^d)^G$ and there exists a polynomial $Q \in \mathbb{R}[Y_1,\ldots, Y_m]$ such that $\phi = Q(p_1, \ldots, p_m) = Q \circ \pi$. Now, let $K$ be a compact set in $\mathbb{R}^m$. Then $Q(K)$ is a compact of $\mathbb{R}$ and $\phi^{-1}(Q(K))$ is a compact set of $\mathbb{R}^d$ (because $\phi$ is proper). Finally, $\pi^{-1}(K)$ is a closed subset of $\mathbb{R}^d$ (because $\pi$ is continuous) included in $\phi^{-1}(Q(K))$, hence it is compact as well.

The restriction of a proper map being a proper map, the map $\pi : X \rightarrow \pi(X)$ is also proper. 
\\

Since $\mathbb{R}^m$ is a Hausdorff locally compact space, the proper map $\pi : \mathbb{R}^d \rightarrow \mathbb{R}^m$ is closed. Now, consider a closed subset $F \cap X$ of $X$, where $F$ is a closed subset of $\mathbb{R}^d$. Since $X$ is stable under the action of $G$, $\pi(F \cap X) = \pi(F) \cap \pi(X)$ and this is a closed subset of $\pi(X)$ since $\pi(F)$ is a closed subset of $\mathbb{R}^m$ ($\pi$ is closed on $\mathbb{R}^d$). Consequently, $\pi : X \rightarrow \pi(X)$ is a closed map.

It is open as well. Indeed, take $U$ to be an open subset of $X$, then $\pi(U) = \pi(\bigcup_{g \in G} g \cdot U)$ so we can assume $U$ to be globally stable under the action of $G$. Now, $\pi(X \setminus U)$ is a closed subset of $\pi(X)$, but $\pi(X \setminus U) = \pi(X) \setminus \pi(U)$ (since $U$ is stable under the action of $G$) therefore $\pi(U)$ is an open subset of $\pi(X)$. 
\\

Now, let us show that the map $\pi : X \rightarrow \pi(X)$ is homeomorphic to the topological quotient map $q : X \rightarrow X/G$.

The map $\pi : X \rightarrow \pi(X)$ is a continuous surjective map such that if $x,y \in X$, $\pi(x) = \pi(y)$ if and only if there exists $g \in G$ such that $x = g \cdot y$. As a consequence, there exists a continuous bijective map $\overline{\pi} : X/G \rightarrow \pi(X)$ such that $\pi = \overline{\pi} \circ q$. The map $\overline{\pi}$ is a homeomorphism since $\pi$ is an open (or closed) map. 
\end{proof}

\subsection{$G$-$\mathcal{AS}$-sets} 

In this paragraph, we assume that $X$ is an $\mathcal{AS}$-set $X$ of $\mathbb{P}^d(\mathbb{R})$ in the sense of \cite{KP} (Definition 3.1 and Definition 2.7). We recall below the definitions of an arc-symmetric set of $\mathbb{P}^d(\mathbb{R})$ and of an $\mathcal{AS}$-set of $\mathbb{P}^d(\mathbb{R})$. 

\begin{de} \label{defas} Let $A$ be a subset of $\mathbb{P}^d(\mathbb{R})$.
\begin{itemize}
	\item $A$ is called an arc-symmetric subset of $\mathbb{P}^d(\mathbb{R})$ if $A$ is semialgebraic and if, for every real analytic arc $\gamma : (-1 , 1) \rightarrow \mathbb{P}^d(\mathbb{R})$, the inclusion $\gamma((-1,0)) \subset A$ implies the entire inclusion $\gamma((-1,1)) \subset A$. 
	\item $A$ is called an $\mathcal{AS}$-set $X$ of $\mathbb{P}^d(\mathbb{R})$ if $A$ is Boolean combination of arc-symmetric sets of $\mathbb{P}^d(\mathbb{R})$.
\end{itemize}
\end{de}

\begin{rem} \begin{itemize} 
	\item The arc-symmetric sets of $\mathbb{P}^d(\mathbb{R})$ are closed (with respect to strong topology) hence compact (the arc-symmetric sets of $\mathbb{P}^d(\mathbb{R})$ are the closed $\mathcal{AS}$-sets of $\mathbb{P}^d(\mathbb{R})$).
	\item A set $A \subset \mathbb{P}^d(\mathbb{R})$ is an $\mathcal{AS}$-set if and only if, for every real analytic arc $\gamma : (-1 , 1) \rightarrow \mathbb{P}^d(\mathbb{R})$ such that $\gamma((-1,0)) \subset A$, there exists $\epsilon > 0$ such that $\gamma((0, \epsilon)) \subset A$.
	\item The real algebraic sets, or more generally any Zariski open subset of a real algebraic set, are $\mathcal{AS}$-sets, as well as their compact connected components. 
\end{itemize}
\end{rem}

We will suppose furthermore that the projective Zariski closure of $X$ in $\mathbb{P}^d(\mathbb{R})$ is equipped with an action of $G$ by biregular isomorphisms, and that $X$ is globally stabilized under this action : we will say that $X$ is a $G$-$\mathcal{AS}$-set (see \cite{GF}, paragraph 3.1.1).

Now, recall that the real algebraic variety $\mathbb{P}^d(\mathbb{R})$ can be biregularly embedded into a compact algebraic subset of $\mathbb{R}^{(d+1)^2}$ (\cite{BCR}, Theorem 3.4.4), so that $X$ can be supposed to be, up to equivariant biregular isomorphism, an $\mathcal{AS}$-subset of $\mathbb{R}^{(d+1)^2}$ such that its (affine) Zariski closure (in $\mathbb{R}^{(d+1)^2}$) is compact and equipped with an action of $G$ by biregular isomorphisms which globally preserves $X$ ($X$ is a boolean combination of compact arc-symmetric sets of $\mathbb{R}^{(d+1)^2}$ : see Remark 3.5 of \cite{KP}). For the sake of simplicity, let us denote again $\mathbb{R}^{(d+1)^2}$ by $\mathbb{R}^d$.

Moreover, since $G$ is finite, we can suppose, again up to equivariant biregular isomorphism, that the action of $G$ on $\overline{X}^{\mathcal{Z}}$ is linear, given by permutation of coordinates : use the biregular analog of lemma \ref{linearisationfinigroup}.

\begin{rem} A real algebraic set (or more generally a Zariski open subset of a real algebraic set) equipped with an action of $G$ via biregular isomorphisms is equivariantly biregularly isomorphic to a $G$-$\mathcal{AS}$-set : use the biregular analog of lemma \ref{linearisationfinigroup} and extend the action on $\mathbb{R}^D$ by permutation matrices to an action on $\mathbb{P}^D(\mathbb{R})$. Such a set will also be called a $G$-$\mathcal{AS}$-set and we will implicitly confound it with its isomorphic image. 
\end{rem} 
 
We can then consider the image of the semialgebraic set $X$ under the quotient map $\pi : \mathbb{R}^d \rightarrow Y \subset \mathbb{R}^m$ (the action of $G$ on $\overline{X}^{\mathcal{Z}}$ is induced from the action on $\mathbb{R}^d$ : lemma \ref{quotientinabigger}).

We will now recall the proof of an important result of \cite{GF} (Proposition 3.15) : if we suppose $X$ to be compact and that the action of $G$ on $X$ is free, then $\pi(X)$ is also a compact arc-symmetric set :

\begin{prop} \label{quotientarcsym} Suppose that $X$ is compact (in particular, $X$ is a compact arc-symmetric subset of $\mathbb{R}^d$) and that, for all $x \in X$, $G_x = \{e\}$. Then $\pi(X)$ is a compact arc-symmetric subset of $\mathbb{R}^m$.
\end{prop}

\begin{proof} $\pi(X)$ is a semialgebraic subset of $\mathbb{R}^m$ as the image of a semialgebraic set by the polynomial map $\pi$. It is furthermore compact since $X$ is compact and $\pi$ is continuous.

Now, considering the standard embedding $\mathbb{R}^m \subset \mathbb{P}^m(\mathbb{R})$, we are going to prove that $\pi(X)$ is an $\mathcal{AS}$-set of $\mathbb{P}^m(\mathbb{R})$. Consider a real analytic map $\gamma : (-1,1) \rightarrow \mathbb{P}^m(\mathbb{R})$ such that $\gamma((-1,0)) \subset \pi(X)$. First, notice that $y := \gamma(0) \in \pi(X)$ because $\pi(X)$ is closed. Let $x \in X \subset \mathbb{R}^d$ such that $y = \pi(x)$. Since $x$ is a nonsingular point of $\mathbb{R}^d$ and $G_x = \{e\}$, by proposition \ref{localdiff}, there exists a semialgebraic open neighborhood $U$ of $x$ in $\mathbb{R}^d$ and a semialgebraic open neighborhood $U'$ of $y = \pi(x)$ in the Zariski closure $W$ of $Y$, such that $\pi_{|U}$ is a Nash diffeomorphism from $U$ to $U'$. Denote by $\eta$ the inverse map.

Let $\epsilon' > 0$ such that $\gamma((-\epsilon', 0]) \subset U'$. Composing $\gamma_{|(-\epsilon', 0]}$ with $\eta$, we obtain an analytic map $\eta \circ \gamma : (-\epsilon', 0] \rightarrow U \subset \mathbb{R}^d$ that we can extend into an analytic map $\widetilde{\gamma} : (-\epsilon', \epsilon) \rightarrow U$, with $0<\epsilon < 1$. Since $\gamma((-\epsilon',0]) \subset \pi(X)$, we have $\widetilde{\gamma}((-\epsilon',0]) = \eta \circ \gamma((-\epsilon',0]) \subset X$ and, since $X$ is arc-symmetric, $\widetilde{\gamma}((-\epsilon',\epsilon)) \subset X$.
 
We finally apply $\pi$ to obtain a real analytic arc $\pi \circ \widetilde{\gamma} : (-\epsilon', \epsilon) \rightarrow \mathbb{R}^m  \subset \mathbb{P}^m(\mathbb{R})$ which coincides with the real analytic arc $\gamma : (-1,1) \rightarrow \mathbb{P}^m(\mathbb{R})$ on $(-\epsilon', 0)$. As a consequence, by analytic continuation, $\gamma((-\epsilon', \epsilon) ) = \pi \circ \widetilde{\gamma}((-\epsilon', \epsilon)) \in \pi(X)$. As a conclusion, $\pi(X)$ is an $\mathcal{AS}$-set of $\mathbb{P}^m(\mathbb{R})$.

Since it is compact, it is an arc-symmetric subset of $\mathbb{P}^m(\mathbb{R})$ and, since $\pi(X) \subset \mathbb{R}^m$, $\pi(X)$ is an arc-symmetric subset of $\mathbb{R}^m$ (Remark 3.5 of \cite{KP}).
\end{proof}

\begin{rem} \begin{enumerate}
	\item Under the same hypotheses, if we suppose $X$ to be an algebraic set, the quotient $\pi(X)$ is not algebraic in general. Consider the example of Remark 3.16 of \cite{GF} : the quotient of the compact algebraic subset $\left\{y^2 + (x^2-2)(x^2-1)(x^2+1) = 0\right\}$ of $\mathbb{R}^2$ by the free involution $(x,y) \mapsto (-x,y)$ is the compact connected component $\{x \geq 0\}$ of the algebraic set $\left\{y^2 + (x-2)(x-1)(x+1) = 0\right\}$ (it is non-algebraic arc-symmetric set).
	\item If the action of $G$ on $X$ is free but $X$ is not compact, the quotient $\pi(X)$ is not an $\mathcal{AS}$-set in general. Consider the third example of \ref{secexquotient} : the half-hyperbola is not an $\mathcal{AS}$-set (see also remark \ref{remdifemb} below). 
\end{enumerate}
\end{rem}

We can actually establish a slightly more general result. If $A$ is a subset of $\mathbb{P}^d(\mathbb{R})$, denote by $\overline{A}^{\mathcal{AS}}$ the smallest arc-symmetric set of $\mathbb{P}^d(\mathbb{R})$ containing $A$ (it is the intersection of all arc-symmetric subsets of $\mathbb{P}^d(\mathbb{R})$ containing $A$). If $A \subset \mathbb{R}^d$ and if the Zariski closure $\overline{A}^{\mathcal{Z}}$ of $A$ in $\mathbb{R}^d$ is compact, then $\overline{A}^{\mathcal{AS}} \subset \overline{A}^{\mathcal{Z}}$.

\begin{cor} \label{quotientas} Suppose that the action of $G$ on the $\mathcal{AS}$-closure $\overline{X}^{\mathcal{AS}}$ of $X$ is free. Then $\pi(X)$ is an $\mathcal{AS}$-set of $\mathbb{P}^m(\mathbb{R})$ contained in $\mathbb{R}^m$ (in the following, we will simply say $\mathcal{AS}$-set of $\mathbb{R}^m$).
\end{cor}

\begin{proof} First, notice that the action of $G$ on $\overline{X}^{\mathcal{Z}}$ globally stabilizes $\overline{X}^{\mathcal{AS}}$ (because it stabilizes $X$ and the image of a compact $\mathcal{AS}$-set by a regular isomorphism is a compact $\mathcal{AS}$-set).

Now, even if $X$ is not anymore supposed to be compact, $\overline{X}^{\mathcal{AS}}$ is, and, since for all $x \in \overline{X}^{\mathcal{AS}}$, $G_x = \{e\}$, $\pi\left(\overline{X}^{\mathcal{AS}}\right)$ is a compact arc-symmetric subset of $\mathbb{R}^m$.

We then proceed by induction on the dimension of $X$ using that $\dim\left( \overline{X}^{\mathcal{AS}} \setminus X\right) < \dim(X)$ (Proposition 3.3 of \cite{KP}), and that $\pi\left( \overline{X}^{\mathcal{AS}} \setminus X \right) = \pi\left( \overline{X}^{\mathcal{AS}}\right) \setminus \pi(X)$ i.e. $\pi(X) = \pi\left( \overline{X}^{\mathcal{AS}}\right) \setminus \pi\left( \overline{X}^{\mathcal{AS}} \setminus X \right)$. 
\end{proof}

\begin{de} \label{deffreegasset} If the action of $G$ on $\overline{X}^{\mathcal{AS}}$ is free, we will say that $X$ is a free $G$-$\mathcal{AS}$-set.
\end{de}

\subsection{Functoriality of the quotient of a free $G$-$\mathcal{AS}$-set with respect to equivariant continuous maps with $\mathcal{AS}$-graph} \label{functorialityquotientfree}

Consider an equivariant continuous map with $\mathcal{AS}$-graph $f : X \rightarrow X'$ between two free $G$-$\mathcal{AS}$-sets. As in the previous paragraph, we can suppose (because a biregular isomorphism is in particular an analytic isomorphism with semialgebraic graph) that $X$ and $X'$ are $\mathcal{AS}$-sets of some $\mathbb{R}^{d}$ and $\mathbb{R}^{d'}$ respectively, such that their affine Zariski closures are compact, and that the respective actions of $G$ are given by permutations of coordinates.

If $\pi : \mathbb{R}^d \rightarrow Y$ and $\pi' : \mathbb{R}^{d'} \rightarrow Y'$ are the respective geometric quotients of $\mathbb{R}^d$ and $\mathbb{R}^{d'}$ by~$G$, we define a continuous map with $\mathcal{AS}$-graph
$$f_{/G} : \pi(X) \rightarrow \pi'(X')$$
between the $\mathcal{AS}$-sets $\pi(X)$ and $\pi'(X)$, such that the following diagram commutes :
$$\begin{array}{ccc} 
X & \stackrel{f}{\longrightarrow} & X'\\
 ~ \downarrow {\scriptstyle \pi} & & ~ \downarrow {\scriptstyle \pi'}\\
\pi(X) & \stackrel{f_{/G}}{\longrightarrow} & \pi'(X')
\end{array}$$  

Precisely, if $y = \pi(x) \in \pi(X)$, we set $f_{/G}(y) := \pi'(f(x))$ (notice that this definition is independent of the chosen preimage of $y$ since $f$ is equivariant and because the fiber of $\pi$ at $y$ is the orbit of $x$ under the action of $G$ on $X$).

\begin{rem} \label{quotientmappingpoly} If $X$ and $X'$ are algebraic (that is $X = \overline{X}^{\mathcal{Z}}$ and $X' = \overline{X'}^{\mathcal{Z}}$) and $\psi : X \rightarrow X'$ is an equivariant polynomial map, then $\psi_{/G} : \pi(X) \rightarrow \pi'(X')$ is a polynomial map as well (it is the map $\rho$ of lemma \ref{indepemb}).
\end{rem}

We first prove that $f_{/G} : \pi(X) \rightarrow \pi'(X')$ is continuous :

\begin{prop} \label{quotientcont} Suppose that $f$ is a $\mathcal{C}^k$-map with $k \in \mathbb{N}$, $k = \infty$ or $k = \omega$. Then $f_{/G}$ is a $\mathcal{C}^k$-map as well.
\end{prop} 

\begin{proof} Let $y = \pi(x) \in \pi(X)$. Denote by $W$ the Zariski closure of $Y$. Since $x$ is a nonsingular point of $\mathbb{R}^d$ and $G_{x} = \{e\}$, according to proposition \ref{localdiff}, there exist a semialgebraic open neighborhood $U$ of $x$ in $\mathbb{R}^d$ and a semialgebraic open neighborhood $U'$ of $\pi(x)$ in $W$ such that $\pi_{| U}$ is a Nash diffeomorphism from $U$ to $U'$. As a consequence, on the open neighborhood $U' \cap \pi(X)$ of $y$, $f_{/G}$ can be described as $\pi' \circ f \circ \pi_{| U}^{-1}$, hence the result (recall that $\pi'$ is a polynomial map).
\end{proof}

\begin{rem} By the same argument, we can show that, if $f : X \rightarrow X'$ is arc-analytic (that is sends real analytic arcs on real analytic arcs), then so is $f_{/G} : \pi(X) \rightarrow \pi'(X')$.
\end{rem}

We will now show that the graph of $f_{/G}$ is an $\mathcal{AS}$-set, by describing it as the image of a free $(G \times G)$-$\mathcal{AS}$-set by a geometric quotient :

\begin{prop} \label{quotientmappingAS} Suppose $Y \subset \mathbb{R}^m$ and $Y' \subset \mathbb{R}^{m'}$.
The graph of $f_{/G}$ is an $\mathcal{AS}$-set of $\mathbb{R}^{m+m'}$.
\end{prop}

\begin{proof} The graph $\Gamma_{f_{/G}}$ is
$$\Gamma_{f_{/G}} = \left\{\left(\pi(x), \pi'(x')\right) \in Y \times Y' ~|~x \in X, x' \in X', x' = f(x)\right\}.$$
Hence, it is the image of the graph $\Gamma_f$ of $f$ under the geometric quotient $\pi \times \pi' : \mathbb{R}^d \times \mathbb{R}^{d'} \rightarrow Y \times Y'$ of $\mathbb{R}^d \times \mathbb{R}^{d'}$ under the product action of $G \times G$ (lemma \ref{quotientprod}). 

Notice that $\Gamma_f$ is not globally stable under the action of $G \times G$. However, we also have 
$$\Gamma_{f_{/G}} = \pi\left( \bigcup_{(g,g') \in G \times G} \alpha_g \times \alpha_{g'} \left( \Gamma_f\right)\right),$$  
the union $\bigcup_{(g,g') \in G \times G} \alpha_g \times \alpha_{g'} \left( \Gamma_f\right)$ being a free $(G \times G)$-$\mathcal{AS}$-set of $\mathbb{R}^d \times \mathbb{R}^{d'}$ (the arc-symmetric closure of this $\mathcal{AS}$-set is included in $\overline{X}^{\mathcal{AS}} \times \overline{X'}^{\mathcal{AS}}$). We conclude by corollary \ref{quotientas}.

\end{proof}

\begin{rem} \label{quotientmappingproper} If $f$ is a proper map, so is $f_{/G}$. Indeed, let $K$ be a compact subset of $\pi'(X')$, then
$$f_{/G}^{-1}(K) = \pi\left(f^{-1}\left({\pi'}^{-1}(K)\right)\right)$$
is a compact subset of $\pi(X)$ since $\pi'$ is proper (lemma \ref{quotientproper}), $f$ is proper and $\pi$ is continuous.
\end{rem}

The operation which associates to $f$ the continuous map with $\mathcal{AS}$-graph $f_{/G}$ is functorial :

\begin{lem} \label{quotientmappingfunctorial} Let $X''$ be a free $G$-$\mathcal{AS}$-set and let $h : X' \rightarrow X''$ be an equivariant continuous map with $\mathcal{AS}$-graph. 

Then the equivariant continuous composition $h \circ f : X \rightarrow X''$ has $\mathcal{AS}$-graph and 
$$(h \circ f)_{/G} = h_{/G} \circ f_{/G}.$$ 
Furthermore, if $id_X$ denotes the identity map on $X$, $(id_X)_{/G} = id_{\pi(X)}$.
\end{lem}

\begin{proof} Denote by $\pi'' : \mathbb{R}^{d''} \rightarrow Y''$ the geometric quotient of $\mathbb{R}^{d''}$ by $G$. Then, for any $\pi(x) \in \pi(X)$,
$$h_{/G} \circ f_{/G} (\pi(x)) = h_{/G} (\pi'(f(x))) = \pi''(h \circ f(x)).$$
\end{proof}

\begin{rem} \label{remdifemb}
Let $i$ another equivariant regular affine embeddings of the (compact) Zariski closure of $X$. Then the Nash equivariant isomorphism $i : X \rightarrow i(X)$ induces a Nash isomorphism $i_{/G}$ between the respective quotients of $X$ and $i(X)$ by $G$ (use proposition \ref{quotientcont} and recall that the image of a semialgebraic set under a quotient map is semialgebraic). Therefore, the geometric quotient of free $G$-$\mathcal{AS}$-sets is well-defined, that is unique up to Nash isomorphism.

By the same arguments as above, we can prove that an equivariant Nash diffeomorphism (for instance an equivariant biregular isomorphism) between two semialgebraic sets equipped with free actions of $G$ induces a Nash diffeomorphism (which has $\mathcal{AS}$-graph in particular) between their geometric quotients. Therefore, the geometric quotient of semialgebraic sets equipped with a free action of $G$ is unique up to Nash isomorphism (see also Theorem 3.4 of~\cite{KP}). 
\end{rem}

\section{Equivariant homologies and cohomologies} \label{secequivhomcohom}

Keep $G$ to be a finite group. From now on, we want to consider and study invariants of $G$-$\mathcal{AS}$-sets. First, we choose the ``equivariant'' homologies and cohomologies with which we are going to work.

\subsection{Equivariant singular homology and cohomology}

Let $X$ be a topological space on which $G$ acts via homeomorphisms. We consider the equivariant singular homology 
$$H^G_*(X) := H_*(X \times_G E_G, \mathbb{Z}_2)$$
and equivariant singular cohomology 
$$H_G^*(X) := H^*(X \times_G E_G, \mathbb{Z}_2)$$
of $X$ with coefficients in $\mathbb{Z}_2 := \mathbb{Z}/2 \mathbb{Z}$, defined using the Borel construction (\cite{Borel}) :
\begin{itemize}
	\item $H_*(\, \cdot \,) := H_*( \, \cdot \, , \mathbb{Z}_2)$ and $H^*(\, \cdot \,) := H^*(\,\cdot \,, \mathbb{Z}_2)$ stand for the singular homology and cohomology with coefficients in $\mathbb{Z}_2$,
	\item $E_G$ is the total space of a universal principal $G$-bundle $E_G \rightarrow B_G$,
	\item $X \times_G EG$ denotes the topological quotient of the product space $X \times EG$ by the diagonal action of $G$.
\end{itemize}

\begin{rem} \label{remequivsinghom} \begin{enumerate}
	\item The equivariant singular homology and cohomology are independent of the choice of a universal $G$-bundle. 
	\item Let $E$ be a contractible topological space equipped with a free action of $G$. Since $G$ is a finite group, the action of $G$ on $E$ is proper and the quotient map $E \rightarrow E/G$ is a universal principal $G$-bundle.
	\item Since $E_G$ is a contractible space, $X \times E_G$ is equivariantly homotopic to $X$ and $X \times_G E_G = (X \times E_G)/G$ is called the homotopy quotient of $X$ by $G$.
	\item The equivariant map $X \rightarrow \{pt\}$ induces a fibration $X \times_G E_G \rightarrow \{pt\} \times E_G/G = B_G$ with fiber $X$. In particular, we have Leray-Serre spectral sequences
	$$E_{p,q}^2 = H_p(B_G, H_q(X)) \Rightarrow H_{p+q}^G(X)$$
and  
	$$E^{p,q}_2 = H^p(B_G, H^q(X)) \Rightarrow H^{p+q}_G(X)$$
(see for instance \cite{McCleary} section 11.4).
	\item If $G = \{e\}$, $H^G_*(X) = H_*(X)$. If $X$ is contractible (e.g. if $X$ is a point), $H^G_*(X) = H_*^G(B_G) = H_*(G , \mathbb{Z}_2)$, since $G$ is finite, where $H_*(G , \mathbb{Z}_2)$ is the homology of the group $G$ with coefficients in $\mathbb{Z}_2$ (see for instance \cite{Brown}). If the action of $G$ on $X$ is trivial, $H_*^G(X) = H_*(X) \otimes_{\mathbb{Z}_2} H_*(G , \mathbb{Z}_2)$ (in this case, $X \times_G E_G = X \times B_G$ and we use the K\"unneth isomorphism). Finally, if the action of $G$ on $X$ is free, we have $H^G_*(X) = H_*(X/G)$ (in this case, the equivariant map $E_G \rightarrow \{pt\}$ induces a fibration $X \times_G E_G \rightarrow X/G$ with fiber $E_G$, which is contractible). We have similar statements for the cohomological counterparts.
	
	\item If $X$ is a $G$-$CW$-complex, i.e. if $X$ is a $CW$-complex and the action of $G$ permutes its cells, the equivariant singular homology and cohomology coincide respectively with the equivariant homology and cohomology defined in \cite{Brown} Chap. VII, sect. 7. Indeed, consider a contractible $G$-$CW$-complex $E$ such that $G$ freely permutes its cells (see for instance the construction of \cite{Hatcher}  Example 1B.7). We have an equivariant chain isomorphism
$$C_*^{cell}(E) \otimes_{\mathbb{Z}_2} C_*^{cell}(X) \rightarrow C_*^{cell}(E \times X),$$
where $C_*^{cell}(\, \cdot \,)$ denotes the cellular chain complex with coefficients in $\mathbb{Z}_2$, which induces a chain isomorphism
$$C_*^{cell}(E) \otimes_{G} C_*^{cell}(X) \rightarrow C_*^{cell}(E \times X)_{G} \rightarrow C_*^{cell}\left((E \times X)/G\right)$$  
(see \cite{Brown} Chap. III sect. 0 and Chap. II Proposition 2.4), which itself induces an isomorphism in equivariant homology (the cellular complex of the homotopy quotient $(E \times X)/G$ computes its homology, and $C^{cell}_*(E)$ is a free resolution of $\mathbb{Z}$ over $\mathbb{Z}[G]$ since $E$ is a free $G$-$CW$-complex : see \cite{Brown} Chap. I Proposition 4.1). 

As for the cohomological counterpart, apply the duality functor $\Hom_{\mathbb{Z}_2}( \, \cdot \, , \mathbb{Z}_2)$ to the above chain isomorphism and use the tensor-hom adjunction to obtain natural isomorphisms
$$\Hom_{\mathbb{Z}_2}\left(C_p^{cell}(E) \otimes_{G} C_q^{cell}(X) , \mathbb{Z}_2\right) \cong \Hom_G\left(C_p^{cell}(E), \Hom_{\mathbb{Z}_2}\left(C_q^{cell}(X), \mathbb{Z}_2\right)\right).$$ 

\end{enumerate}
\end{rem}

\subsection{A real algebraic model for the total space $E_G$}

In order to study the equivariant geometry of $G$-$\mathcal{AS}$-sets via these equivariant homology and cohomology, we choose for $E_G$ a convenient ``real algebraic'' model : for $n \in \mathbb{N} \setminus \{0\}$ and $k \in \mathbb{N} \cup \{\infty\}$, consider the Stiefel manifold $V_n(\mathbb{R}^k)$, which is the set of the orthonormal $n$-frames of $\mathbb{R}^k$, that is $n$-tuples of orthonormal vectors of $\mathbb{R}^k$ (recall that $\mathbb{R}^{\infty}$ denotes the space of sequences $(x_n)_{n \in \mathbb{N} \setminus \{0\}}$ such that $x_n \neq 0$ for finitely many $i$'s).

Remark that, if $k$ is finite, $V_n(\mathbb{R}^k)$ is a compact space, as closed subspace of the product of $n$ copies of the unit sphere in $\mathbb{R}^k$. Besides, the natural inclusions $V_n(\mathbb{R}^k) \hookrightarrow V_n(\mathbb{R}^{k+1}) \hookrightarrow V_n(\mathbb{R}^{\infty})$ fit into the equality
$$V_{n}(\mathbb{R}^{\infty}) = \varinjlim_{k \in \mathbb{N}} V_{n}(\mathbb{R}^k)$$ 
(it is an infinite increasing union in $V_{n}(\mathbb{R}^{\infty})$).

Furthermore, $V_{n}(\mathbb{R}^{\infty})$ is a contractible space (see for instance Example 4.53 of \cite{Hatcher}) and one can equip $V_{n}(\mathbb{R}^{\infty})$, as well as any $V_n(\mathbb{R}^k)$, with a free action of $G$ : 

\begin{lem} Let $n \in \mathbb{N} \setminus \{0\}$ and $k \in \mathbb{N} \cup \{\infty\}$, and let $H$ be a subgroup of the orthogonal group $O_n(\mathbb{R})$. There is a free action of $H$ on $V_n(\mathbb{R}^k)$.
\end{lem}

\begin{proof} Identify $V_n(\mathbb{R}^k)$ with the set of matrices $A \in M_{k,n}(\mathbb{R})$ such that ${}^t\!A A = I_n$ and, if $M \in H$ and $A \in V_n(\mathbb{R}^k)$, set $M \cdot A := A M^{-1} \in V_n(\mathbb{R}^k)$ (we have ${}^t\!(A M^{-1}) A M^{-1} = {}^t\!(M^{-1}) {}^t\!A A M^{-1} = I_n$).  

This action of $H$ on $V_n(\mathbb{R}^k)$ is free since, if $M \in H$, $A \in V_n(\mathbb{R}^k)$, $A M = A \Rightarrow M = I_n$.
\end{proof}

\begin{rem} \label{remactionstiefel}
\begin{itemize} 
	\item We can describe the action of $H$ on $V_n(\mathbb{R}^k)$ as a restriction of a linear action of $H$ on $\left(\mathbb{R}^n\right)^k$. Indeed, if $A = (v_1, \ldots, v_n) \in V_n(\mathbb{R}^k)$, write $v_i = (a_{1,i}, \ldots, a_{k,i})$ and associate to $(v_1, \ldots, v_n)$ the single vector 
$$v := (a_{1,1}, \ldots, a_{1,n}, \ldots, a_{k,1}, \ldots, a_{k,n}).$$
Now, let $M$ be an element of $H$ and $M_0$ be the block diagonal matrix with $k$ copies of $M$ as diagonal blocks. Then, via this correspondence, the action of $M$ on $A$ corresponds to the (left) application of the matrix $M_0$ to the vector $v \in \left(\mathbb{R}^n\right)^k$ (recall that $M^{-1} = {}^t\!M$).

	\item If $|G| = N$, we can embed the finite group $G$ into the orthogonal group $O_N(\mathbb{R})$ via permutation matrices.
\end{itemize}
\end{rem} 

As a consequence, the quotient map $V_{n}(\mathbb{R}^{\infty}) \rightarrow V_{n}(\mathbb{R}^{\infty})/G$ is a universal principal $G$-bundle. Moreover :

\begin{lem}
For any $n \in \mathbb{N} \setminus \{0\}$ and $k \in \mathbb{N}$, the Stiefel manifold $V_n(\mathbb{R}^k)$ is a compact nonsingular algebraic subset of $\mathbb{R}^{nk}$.
\end{lem}

\begin{proof} The set $V_n(\mathbb{R}^k) \subset \left(\mathbb{R}^{k}\right)^n$ is described by the real algebraic equations
$$P_i := a_{1,i}^2 + \ldots + a_{k,i}^2 - 1 = 0,~~1\leq i \leq n$$
and
$$Q_{i,j} := a_{1,i} a_{1,j} + \ldots + a_{k,i} a_{k,j} = 0,~~1 \leq i < j \leq n.$$

Now, the columns of the matrix 
$$
\setcounter{MaxMatrixCols}{20}
\begin{pmatrix}
2 a_{1,1} & & & & a_{1,2} & a_{1,3} & \cdots & a_{1,n} & & &&&&&\\
\vdots & & & & \vdots & \vdots & & \vdots & & &&&&&\\
2 a_{k,1} & & & & a_{k,2} & a_{k,3} & \cdots & a_{k,n} & &&&&&&\\
& 2 a_{1,2} & & & a_{1,1} & & & & a_{1,3} & \cdots & a_{1,n}&&&&\\
& \vdots &&  & \vdots & & && \vdots &&\vdots &&&&\\
& 2 a_{k,2} & & & a_{k,1} & & & & a_{k,3} & \cdots &a_{k,n}&&&&\\
& & \ddots & & &\ddots  & & & &&&&\ddots&&\\
& & & 2 a_{1,n} & && & a_{1,1} & && a_{1,2}&& &&a_{1,n-1}\\
& & & \vdots & & & & \vdots & &\ddots & \vdots &&&&\vdots\\
& & & 2 a_{k,n} & &&& a_{k,1} & && a_{k,2} && &&a_{k, n-1}
\end{pmatrix}$$
of the partial derivatives of the polynomials $P_i$'s and $Q_{i,j}$'s, at any point $(a_{i,j})_{i,j}$ of $V_n(\mathbb{R}^k)$, are linearly independent (otherwise we would have a nontrivial linear relation between the orthonormal vectors $(a_{1,j}, \ldots, a_{n,j})$, $1\leq j \leq n$, which is impossible), so that the real algebraic set $V_n(\mathbb{R}^k)$ is nonsingular of dimension $nk - \frac{n(n+1)}{2}$. 
\end{proof}

Denote $N := |G|$ and suppose that $X$ is a $G$-$\mathcal{AS}$-set. We are going to realize the equivariant singular homology 
$$H_*^G(X) = H_*\left(X \times_G V_{N}(\mathbb{R}^{\infty})\right)$$
of $X$ as an inductive limit of the singular homologies of the geometric quotients of $X \times V_N(\mathbb{R}^k)$ by $G$.
\\

First, recall (section \ref{sectionquotientAS}) that we can assume $X$ to be an $\mathcal{AS}$-subset of $\mathbb{R}^d$ with compact Zariski closure (in $\mathbb{R}^d$) such that $X$ (and then $\overline{X}^{\mathcal{Z}}$) is globally stabilized under a linear action of $G$ on $\mathbb{R}^d$. 

Now, let $k \in \mathbb{N}$. The Stiefel manifold $V_N(\mathbb{R}^k)$ is acted by a linear action of $G$ on $\mathbb{R}^{N k}$ (remark \ref{remactionstiefel}), and $X \times V_N(\mathbb{R}^k)$ is an $\mathcal{AS}$-subset of $\mathbb{R}^{d+N k}$ globallly stabilized under the induced diagonal linear action of $G$ on $\mathbb{R}^{d+N k}$. 

Denote $N_k := d+N k$ and let $p_1, \ldots, p_m$ be generators of the corresponding invariant algebra $\mathbb{R}[X_1, \ldots, X_{N_k}]^G$. The action of $G$ on $X \times V_N(\mathbb{R}^{k+1}) \subset \mathbb{R}^{N_{k+1}}$ is the diagonal action of $G$ on $\mathbb{R}^{N_k} \times \mathbb{R}^N$ (remark \ref{remactionstiefel}). We can then suppose the generators of the invariant algebra $\mathbb{R}[X_1, \ldots, X_{N_k}, X_{N_k + 1}, \ldots, X_{N_k+N}]^G$ to be the polynomials $p_1, \ldots, p_m$ together with polynomials $q_1, \ldots, q_{m'} \in \mathbb{R}[X_1, \ldots, X_{N_k}, X_{N_k + 1}, \ldots, X_{N_k+N}]$ such that $q_j(X_1, \ldots, X_{N_k}, 0, \ldots, 0) = 0$ for $j \in \{1, \ldots, m'\}$. 

If we denote by $i$ the natural embedding of $\mathbb{R}^{N_k}$ in $\mathbb{R}^{N_k + N}$, it induces by functoriality of the real geometric quotient (lemma \ref{indepemb}) the following commutative diagram
$$\begin{array}{lcl} 
\mathbb{R}^{N_k} & \stackrel{i}{\longrightarrow} & \mathbb{R}^{N_{k+1}}\\
 \downarrow {\scriptstyle \pi_k = (p_1, \ldots, p_m)} & &  \downarrow {\scriptstyle \pi_{k+1} = (p_1, \ldots, p_m, q_1, \ldots, q_{m'})}\\
Y \subset \mathbb{R}^m & \stackrel{\rho}{\longrightarrow} & Y' \subset \mathbb{R}^{m+m'}
\end{array}$$
of geometric quotients, where $\rho : \begin{array}{ccc} Y & \rightarrow & Y' \\ (y_1, \ldots, y_m) &\mapsto & (y_1, \ldots, y_m, 0, \ldots, 0) \end{array}$.
Notice that, if we denote by $y_1, \ldots, y_{m}, z_1, \ldots, z_{m'}$ the coordinates in $\mathbb{R}^{m+m'}$, $\rho(Y) = Y' \cap \{ z_1 = \ldots = z_{m'} = 0 \}$.

Finally, we restrict to the inclusion $i : X \times V_N(\mathbb{R}^{k}) \hookrightarrow X \times V_N(\mathbb{R}^{k+1})$ to obtain the commutative diagram
$$\begin{array}{ccc} 
X \times V_N(\mathbb{R}^{k}) & \stackrel{i}{\longrightarrow} & X \times V_N(\mathbb{R}^{k+1})\\
 ~ \downarrow {\scriptstyle \pi_k} & & ~ \downarrow {\scriptstyle \pi_{k+1}}\\
~~~~~\pi_k\left(X \times V_N(\mathbb{R}^{k})\right) \subset \mathbb{R}^m & \stackrel{\rho}{\longrightarrow} & ~~~~\pi_{k+1}\left(X \times V_N(\mathbb{R}^{k+1})\right) \subset \mathbb{R}^{m+m'}
\end{array}$$
In particular, $\rho\left(\pi_k\left(X \times V_N(\mathbb{R}^{k})\right)\right) = \pi_k\left(X \times V_N(\mathbb{R}^{k+1})\right) \cap \{ z_1 = \ldots = z_{m'} = 0 \}$. Remark that the $\mathcal{AS}$-closure of $X \times V_N(\mathbb{R}^k)$ is $\overline{X}^{\mathcal{AS}} \times V_N(\mathbb{R}^k)$ (use an induction of dimension together with   Proposition 3.3 of \cite{KP}), on which the diagonal action of $G$ is free. As a consequence, the geometric quotient $\pi_k\left(X \times V_N(\mathbb{R}^k)\right)$ of $X \times V_N(\mathbb{R}^k)$ by $G$ is an $\mathcal{AS}$-set of $\mathbb{R}^m$ (corollary \ref{quotientas}). 
\\

The maps $\pi_k : X \times V_N(\mathbb{R}^{k}) \rightarrow \pi_{k}\left(X \times V_N(\mathbb{R}^{k})\right)$ form an inductive system which induces a map 
$$\Pi : X \times V_N(\mathbb{R}^{\infty})  = \varinjlim_{k \in \mathbb{N}} X \times V_N(\mathbb{R}^{k}) \rightarrow \varinjlim_{k \in \mathbb{N}} \pi_{k}\left(X \times V_N(\mathbb{R}^{k})\right).$$

Since each map $\pi_k$ is continuous, closed, open and verifies, for $x,y \in V_N(\mathbb{R}^{k})$, $\pi_k(x) = \pi_k(y)$ if and only there exists $g \in G$ such that $x = g \cdot y$, the map $\Pi$ is continuous, closed, open and verifies for $x,y \in V_N(\mathbb{R}^{\infty})$, $\Pi(x) = \Pi(y)$ if and only there exists $g \in G$ such that $x = g \cdot y$ as well.

In particular, the surjective map
$$\Pi : X \times V_N(\mathbb{R}^{\infty}) \rightarrow \varinjlim_{k \in \mathbb{N}} \pi_{k}\left(X \times V_N(\mathbb{R}^{k})\right)$$
is (up to homeomorphism) the topological quotient map $X \times V_N(\mathbb{R}^{\infty}) \rightarrow \left(X \times V_N(\mathbb{R}^{\infty})\right)/G$. As a consequence,
$$H_*^G(X) = H_*\left( \varinjlim_{k \in \mathbb{N}} \pi_{k}\left(X \times V_N(\mathbb{R}^{k})\right)\right).$$

We then use Proposition 3.33 of \cite{Hatcher} to establish the following statement :

\begin{prop} \label{equivsinghomindlim1} Denote $X_{(k)} := \pi_{k}\left(X \times V_N(\mathbb{R}^{k})\right)$. The inductive system of inclusions $X_{(k)} \rightarrow X_{(k+1)}$ induces an isomorphism
$$\varinjlim_{k \in \mathbb{N}} H_*(X_{(k)}) \rightarrow H_*^G(X).$$
\end{prop}

\begin{proof} The inductive limit $\varinjlim_{k \in \mathbb{N}} X_{(k)}$ can be considered as an increasing union in $\mathbb{R}^{\infty}$. Let $K$ be a compact of $\varinjlim_{k \in \mathbb{N}} X_{(k)}$. Then $K$ is a compact of $\mathbb{R}^{\infty}$. 

Now, if $l \in \mathbb{N}$, consider the $CW$-complex structure on $\mathbb{R}^l$ whose $0$-dimensional cells are the points of $\mathbb{Z}^l$ and the higher dimensional open cells are the open hypercubes of edge-length one and vertices in $\mathbb{Z}^l$. These $CW$-complex structures on $\mathbb{R}^l$, $l \in \mathbb{N}$, are compatible with the natural inclusions $\mathbb{R}^l \rightarrow \mathbb{R}^{l+1}$, so that $\mathbb{R}^l$ is a subcomplex of $\mathbb{R}^{l+1}$. Moreover, via these inclusions, they induce a $CW$-complex structure on $\mathbb{R}^{\infty}$, such that each $\mathbb{R}^l$ is a subcomplex of $\mathbb{R}^{\infty}$.

By, for instance, Proposition A.1 of \cite{Hatcher}, the compact set $K$ is then included in a finite subcomplex of $\mathbb{R}^{\infty}$. Therefore, it is included in some $\mathbb{R}^l$ for $l \in \mathbb{N}$ and there exists $k_0 \in \mathbb{N}$ such that $K \subset \mathbb{R}^{m_{k_0}}$, where $\pi_{k_0} = (p_1, \ldots, p_{m_{k_0}})$. Finally,
$$K \subset \varinjlim_{k \in \mathbb{N}} X_{(k)} \cap \mathbb{R}^{m_{k_0}} = X_{(k_0)}$$
and we conclude by Proposition 3.33 of \cite{Hatcher}.

\end{proof}

Dually, we obtain :

\begin{cor} \label{equivsinghomindlim2} The inductive system of inclusions $X_{(k)} \rightarrow X_{(k+1)}$ induces a projective system in singular cohomology $H^*(X_{(k+1)}) \rightarrow H^*(X_{(k)})$ and an isomorphism
$$H^*_G(X) \rightarrow \varprojlim_{k \in \mathbb{N}} H^*(X_{(k)}).$$
\end{cor}

\begin{proof}
We have 
\begin{eqnarray*}
H^*_G(X) & \cong & \Hom_{\mathbb{Z}_2} \left(H_*^G(X), \mathbb{Z}_2\right)\\
		& \xrightarrow{\cong} & \Hom_{\mathbb{Z}_2} \left(\varinjlim_{k \in \mathbb{N}} H_*(X_{(k)}), \mathbb{Z}_2\right) \\
		& = & \varprojlim_{k \in \mathbb{N}} \Hom_{\mathbb{Z}_2} \left(H_*(X_{(k)}), \mathbb{Z}_2\right) \\
		& \cong & \varprojlim_{k \in \mathbb{N}} H^*(X_{(k)})
\end{eqnarray*}
(see also Proposition 3F.5 of \cite{Hatcher} and its proof).
\end{proof}

\subsection{Equivariant homology and cohomology with closed supports}

We now define another equivariant homology for the $G$-$\mathcal{AS}$-set $X$ using the semialgebraic chain complexes with closed supports and coefficients in $\mathbb{Z}_2$ (see \cite{MCP}, Appendix) of the $\mathcal{AS}$-sets $X_{(k)}$, $k \in \mathbb{N}$. Precisely, the natural inclusions $X_{(k)} \rightarrow X_{(k+1)}$ induce an inductive system of injective chain morphisms
$$C_*\left(X_{(k)}\right) \rightarrow C_*\left(X_{(k+1)}\right)$$
and we denote by $C_*(X ; G)$ its inductive limit.

As for the cohomological counterpart, the inclusions $X_{(k)} \rightarrow X_{(k+1)}$ induce a projective system of surjective cochain morphisms
$$C^*\left(X_{(k+1)}\right) \rightarrow C^*\left(X_{(k)}\right),$$
where $C^*$ denote the dual cochain complex of $C_*$ (see \cite{LP}, section 2.3), and we denote by $C^*(X ; G)$ its projective limit.

\begin{de} We define by
$$H_*(X ; G) := H_*(C_*(X ; G))$$
and 
$$H^*(X ; G) := H^*(C^*(X ; G))$$
the respective equivariant homology and cohomology of $X$ with closed supports.
\end{de}

\begin{rem} \begin{itemize} 
\item Because homology commutes with inductive limits (see for instance \cite{Spanier}  Chap. 4, Sec. 1, Theorem 7) and since the semialgebraic chain complex with closed supports computes Borel-Moore homology, we have
$$H_*(X ; G) =  \varinjlim_{k \in \mathbb{N}} H^{BM}_*\left(X_{(k)}\right).$$
\item We also have
$$H^*(X ; G) = \Hom_{\mathbb{Z}_2} \left(H_*(X ; G) , \mathbb{Z}_2\right) = \varprojlim_{k \in \mathbb{N}} H^*_c\left(X_{(k)}\right) :$$
the left $\Hom$ functor of an inductive limit is the projective limit of the $\Hom$ functors (see for instance the proof of Proposition 3F.5 of \cite{Hatcher}), and the dual semialgebraic chain complex computes the cohomology with compact supports (see \cite{LP}, section 2.3).
\item The equivariant homology with closed supports $H_*(X ; G)$ of $X$ is different from the one considered in \cite{Pri-EWF} and \cite{Pri-CPEWF}. When $X$ is compact, the equivariant cohomology with closed supports $H^*(X ; G)$ of $X$ coincides with the equivariant cohomology considered in \cite{Pri-CPEWF} : see remark \ref{coincideequivcohom} below.
\end{itemize}
\end{rem}

\begin{lem} \label{coincideequivhomsingclosed} If $X$ is compact, then $H_*(X ; G) = H^G_*(X)$ and $H^*(X ; G) = H_G^*(X)$.
\end{lem}

\begin{proof} If $X$ is compact, so is each quotient set $X_{(k)}$ (as the image of a compact set by a continuous map), and we have isomorphisms $H_*^{BM}\left(X_{(k)}\right) \rightarrow H_*\left(X_{(k)}\right)$, such that the diagrams
$$\begin{array}{ccc} 
H_*^{BM}\left(X_{(k)}\right) & \stackrel{}{\longrightarrow} & H_*^{BM}\left(X_{(k+1)}\right) \\
 ~ \downarrow {\scriptstyle \cong} & & ~ \downarrow {\scriptstyle \cong}\\
H_*\left(X_{(k)}\right) & \stackrel{}{\longrightarrow} & H_*\left(X_{(k+1)}\right)
\end{array}$$ 
are commutative. As a consequence, the inductive systems $H_*^{BM}\left(X_{(k)}\right)$, $k \in \mathbb{N}$ and $H_*\left(X_{(k)}\right)$, $k \in \mathbb{N}$ are isomorphic and the induced direct limits are isomorphic.

Application of the duality functor $\Hom_{\mathbb{Z}_2}(\, \cdot \, , \mathbb{Z}_2)$ provides the isomorphism $H^*(X ; G) \cong H_G^*(X)$.  
\end{proof}

\begin{rem} \label{coincideequivcohom} By \cite{ParkSuh}, $X$ has a (unique) semialgebraic $G$-$CW$-structure so that, if $X$ is compact, $H_*(X ; G)$ coincides with the equivariant cohomology of $\cite{Brown}$ Chap. VII, sect. 7 (see remark \ref{remequivsinghom} (6)), that is the homology $H_*(G , C_*^{cell}(X))$ of the group $G$ with coefficients in the chain complex $C_*^{cell}(X)$. 

Furthermore, since $G$ is a finite group, $X$ admits a (unique) $G$-equivariant semialgebraic triangulation (which induces its semialgebraic $G$-$CW$-structure) and we can use it to relate the chain complexes $C_*(X)$ and $C_*^{cell}(X)$ via an equivariant quasi-isomorphism. Consequently, if $X$ is compact, $H_*(X ; G) = H_*(G, C_*(X))$ (\cite{Brown} Chap. VII, Proposition 5.2), and by dualization, $H^*(X ; G) = H^*(G, C^*(X))$ (this was the definition of the equivariant cohomology considered in \cite{Pri-CPEWF}, Definition 3.3). 
\end{rem}

\begin{ex} We compute the equivariant homology of the real $2$-dimensional unit sphere $X := \mathbb{S}^{2}$ in $\mathbb{R}^{3}$ equipped with two different kind actions of $G := \mathbb{Z}/2\mathbb{Z}$, using the spectral sequence  
$$E_{p,q}^2 = H_p(G, H_q(X)) \Rightarrow H_{p+q}(X ; G),$$
induced by the double complex $F_* \otimes_G C_*(X)$, if $F_*$ is a projective resolution of $\mathbb{Z}_2$ over $\mathbb{Z}_2[G]$ (see \cite{Brown} Chap. VII, see also \cite{Pri-EWF} section 3).

Another equivariant homology of the sphere was computed in \cite{GF} (Example 2.8) and \cite{Pri-EWF} (Example 3.13), and the equivariant cohomology of the circle was computed in \cite{Pri-CPEWF} (Example 3.5).

\begin{enumerate} \item Consider the action of $G$ on $X$ given by the central symmetry $\sigma : (x_1, x_2, x_{3}) \mapsto (- x_1, -x_2, - x_{3})$. The projective resolution of $\mathbb{Z}_2$ over $\mathbb{Z}_2[G]$ we consider is 
$$\mathbb{Z}_2 \leftarrow \mathbb{Z}_2[G] \xleftarrow{1 + \sigma} \mathbb{Z}_2[G] \xleftarrow{1 + \sigma} \mathbb{Z}_2[G] \leftarrow \cdots$$
so that the above spectral sequence is induced by the double complex 
$$C_*(X) \xleftarrow{id_* + \sigma_*} C_*(X) \xleftarrow{id_* + \sigma_*} C_*(X) \leftarrow \cdots$$

We have $E^3 = E^2$ and in order to compute $E^4$, we have to compute the image of a point $p$ by the differential $d^3$. 

Apply $id_* + \sigma_*$ to $p$ to obtain the union of two opposite points : they are the boundary of a half-equator. If we apply $id_* + \sigma_*$ to this half-equator, we obtain an entire equator which is the boundary of an hemisphere. Finally, the sum of this hemisphere with its image by~$\sigma$ is the entire sphere, so that the page $E^4$ of the spectral sequence is

$$\begin{array}{ccccc}
0 & 0 & 0 & 0 & \cdots \\
0 & 0 & 0 & 0 & \cdots \\
\mathbb{Z}_2[p] & \mathbb{Z}_2[p] & \mathbb{Z}_2[p] & 0 & \cdots 
\end{array}$$

As a consequence,

$$H_k(X ; G) = \begin{cases} \mathbb{Z}_2 & \mbox{ if $k = 0, 1, 2,$} \\
					     0 & \mbox{ otherwise.}
			\end{cases}$$
Notice that, since the action on $X$ is free, we could have used the equality $H_*(X ; G) = H_*(X/G) = H_*(\mathbb{P}_2\left(\mathbb{R})\right)$ (remark \ref{remequivsinghom} (5)).

\item If we consider any non-free action of $G$ on $X$, we can represent the $0$-homology of $X$ by a fixed point. The image of this fixed point by $id_*+ \sigma_*$ is $0$, so $E^2 = E^{\infty}$ and
 
$$H_k(X ; G) = \begin{cases} \mathbb{Z}_2 & \mbox{ if $k = 0, 1,$} \\
					     \mathbb{Z}_2 \oplus \mathbb{Z}_2 & \mbox{ if $k \geq 2$.}
			\end{cases}$$
\item More generally, the equivariant homology of the real $d$-dimensional unit sphere $\mathbb{S}^d$ of $\mathbb{R}^{d+1}$ by an action of $G$ (via a biregular involution) is
$$H_k(\mathbb{S}^d ; G) = \begin{cases} \mathbb{Z}_2 & \mbox{ if $0 \leq k \leq d$,} \\
					     0 & \mbox{ otherwise,}
			\end{cases}$$
if the action is free, and
$$H_k(X ; G) = \begin{cases} \mathbb{Z}_2 & \mbox{ if $0 \leq k \leq d-1$,} \\
					     \mathbb{Z}_2 \oplus \mathbb{Z}_2 & \mbox{ if $k \geq d$,}
			\end{cases}$$	
if there is at least one fixed point.
\end{enumerate}

\end{ex}

\section{The equivariant Nash constructible filtrations} \label{sectequivnashconsfil}

In this section, we construct invariants for $G$-$\mathcal{AS}$-sets with respect to equivariant homeomorphisms with $\mathcal{AS}$-graph.

Precisely, for any $G$-$\mathcal{AS}$-set, we begin by constructing a filtration $\mathcal{N}_{\bullet}$ on the equivariant chain complex $C_*(X; G)$ using the Nash constructible filtration of $\cite{MCP}$. This filtered complex $\mathcal{N}_{\bullet} C_*(X ; G)$ is invariant with respect to equivariant homeomorphisms with $\mathcal{AS}$-graph, as well as the induced spectral sequence $E(X; G)$. From this spectral sequence $E(X; G)$, we recover invariants with values in $\mathbb{Z}$ which are additive with respect to equivariant inclusions of $G$-$\mathcal{AS}$-sets, and coincide with equivariant homology on compact nonsingular $G$-$\mathcal{AS}$-sets : we call them the equivariant virtual Betti numbers of $X$. They are different from the equivariant virtual Betti numbers of \cite{GF}.

\subsection{The homological equivariant Nash constructible filtration}

To any $\mathcal{AS}$-set $T$, we can associate its semialgebraic chain complex $C_*(T)$, which we can equip with the (bounded and increasing) Nash constructible filtration $\mathcal{N}_{\bullet} C_*(T)$ (see \cite{MCP} section 3) :
$$0 = \mathcal{N}_{-q-1} C_q(T) \subset \mathcal{N}_{-q} C_q(T) \subset \cdots \subset \mathcal{N}_{-1} C_q(T) \subset \mathcal{N}_{0} C_q(T) = C_q(T).$$
This filtration on chain level induces a filtration on the Borel-Moore homology with $\mathbb{Z}_2$-coefficients of $T$.

The Nash constructible filtration $\mathcal{N} C_*$ is a functor with respect to proper continuous maps with $\mathcal{AS}$ graph. It is additive on closed inclusions so that we can recover from the induced spectral sequence the virtual Betti numbers of $\mathcal{AS}$-sets (\cite{GF-MI}, see also \cite{MCP-VB}). If $T$ is an affine real algebraic variety, the filtered complex $\mathcal{N}_{\bullet} C_*(T)$ induces the weight spectral sequence of $T$ and its weight filtration on Borel-Moore homology (see \cite{MCP} subsection 1C).

The Nash constructible filtration can also be dualized to induce a cohomological Nash filtration on the cochain complex $C^*(T)$ (see \cite{LP}  section 4) : we denote it by $\mathcal{N}^{\bullet} C^*(T)$. The functor $\mathcal{N} C^*$ is contravariant and have the cohomological counterparts of the properties of $\mathcal{N} C_*$.
\\

In this paragraph, we are going to define an equivariant analog of the Nash constructible filtration on the equivariant chain complex $C_*(X;G)$ of any $G$-$\mathcal{AS}$-set $X$.

So let $X$ be a $G$-$\mathcal{AS}$-set. For each $k \in \mathbb{N}$, we consider the Nash constructible filtration 
$$0 = \mathcal{N}_{-q-1} C_q\left(X_{(k)}\right) \subset \mathcal{N}_{-q} C_q\left(X_{(k)}\right) \subset \cdots \subset \mathcal{N}_{-1} C_q\left(X_{(k)}\right) \subset \mathcal{N}_{0} C_q\left(X_{(k)}\right) = C_q\left(X_{(k)}\right).$$
The closed inclusions $X_k \hookrightarrow X_{k+1}$, $k \in \mathbb{N}$ (see subsection 3.2 above), induce an inductive system of injections of filtered chain complexes
$$\mathcal{N}_{\bullet} C_*\left(X_{(k)}\right) \rightarrow \mathcal{N}_{\bullet} C_*\left(X_{(k+1)}\right)$$
(see Theorem 3.6 of \cite{MCP}). We also denote by $\mathcal{N}_{\bullet}$ the induced direct limit filtration on the complex $C_*(X ; G)$. Notice that 
\begin{itemize}
	\item for each $p$ and $q$, $\mathcal{N}_p C_q(X ; G) = \varinjlim_{k \in \mathbb{N}} \mathcal{N}_p C_q\left(X_{(k)}\right)$,
	\item for each $q$, we have
$$0 = \mathcal{N}_{-q-1} C_q(X ; G) \subset \mathcal{N}_{-q} C_q(X ; G) \subset \cdots \subset \mathcal{N}_{-1} C_q(X ; G) \subset \mathcal{N}_{0} C_q(X ; G) = C_q(X ; G),$$
	\item $\mathcal{N}_{\bullet} C_*(X ; G)$ is bounded filtration in the sense of \cite{McCleary} Theorem 2.6.
\end{itemize}

\begin{de} \label{defequivnashcontr} We call the filtered complex $\mathcal{N}_{\bullet} C_*(X ; G)$ the equivariant Nash constructible filtration of $X$.
\end{de}

\begin{rem} The filtered complex $\mathcal{N}_{\bullet} C_*(X ; G)$ is well-defined up to filtered chain complex isomorphism by remark \ref{remdifemb} and the fact that the functorial Nash filtration is invariant under Nash isomorphisms (a Nash isomorphism is in particular a homeomorphism with $\mathcal{AS}$-graph).
\end{rem}

We are now going to show that the operation which associates to any $G$-$\mathcal{AS}$-set its equivariant Nash constructible filtration is an additive and acyclic functor. These properties are induced by the functoriality, the additivity and the acyclicity of the Nash constructible filtration (Theorem 3.6 of \cite{MCP}).

\begin{theo} \label{functequivnashfil}The map which associates to a $G$-$\mathcal{AS}$-set $X$ its equivariant Nash constructible filtration $\mathcal{N}_{\bullet} C_*(X ; G)$ is a functor with respect to equivariant continuous proper maps with $\mathcal{AS}$-graph.
\end{theo}

\begin{proof} Let $X$ and $Y$ be two $G$-$\mathcal{AS}$-sets and let $f : X \rightarrow Y$ be an equivariant continuous proper maps with $\mathcal{AS}$-graph between $X$ and $Y$. 

Let $k \in \mathbb{N}$. We consider the cartesian product map 
$$f \times id : X \times V_{N}(\mathbb{R}^k) \rightarrow X \times V_{N}(\mathbb{R}^k)$$
of $f$ with the identity of $V_{N}(\mathbb{R}^k)$. $f \times id$ is again an equivariant continuous proper map, and has $\mathcal{AS}$-graph as well (the graph of $f \times id$ is isomorphic to the cartesian product of the graph of $f$, which is $\mathcal{AS}$, and the graph of the identity of $V_{N}(\mathbb{R}^k)$, which is algebraic).

Now, since the sets $X \times V_{N}(\mathbb{R}^k)$ and $Y \times V_{N}(\mathbb{R}^k)$ are free $G$-$\mathcal{AS}$-sets, the map $f \times id$ induces a map $f_{(k)} := (f \times id)_{/G} : X_{(k)} \rightarrow Y_{(k)}$ (see subsection \ref{functorialityquotientfree}) which is continuous (proposition \ref{quotientcont}), proper (remark \ref{quotientmappingproper}) and has $\mathcal{AS}$-graph (proposition \ref{quotientmappingAS}). As a consequence, by functoriality of the Nash constructible filtration, it induces a filtered chain map
$${f_{(k)}}_{*} : \mathcal{N}_{\bullet} C_*\left(X_{(k)}\right) \rightarrow \mathcal{N}_{\bullet} C_*\left(Y_{(k)}\right).$$ 
Furthermore, the operation which associates to $f$ the map ${f_{(k)}}_{*}$ is functorial (use lemma \ref{quotientmappingfunctorial}).

By functoriality of the constructions, the maps ${f_{(k)}}_{*}$, $k \in \mathbb{N}$, are compatible with the injections $\mathcal{N}_{\bullet} C_*\left(X_{(k)}\right) \rightarrow \mathcal{N}_{\bullet} C_*\left(X_{(k+1)}\right)$ and $\mathcal{N}_{\bullet} C_*\left(Y_{(k)}\right) \rightarrow \mathcal{N}_{\bullet} C_*\left(Y_{(k+1)}\right)$, $k \in \mathbb{N}$ (see also remark \ref{quotientmappingpoly}) : precisely, we have commutative diagrams
$$\begin{array}{ccc} 
\mathcal{N}_{\bullet} C_*\left(X_{(k)}\right) & \stackrel{}{\longrightarrow} & \mathcal{N}_{\bullet} C_*\left(X_{(k+1)}\right) \\
 ~ \Big\downarrow {\scriptstyle {f_{(k)}}_{*}} & & ~ \Big\downarrow {\scriptstyle {f_{(k+1)}}_{*}}\\
\mathcal{N}_{\bullet} C_*\left(Y_{(k)}\right) & \stackrel{}{\longrightarrow} & \mathcal{N}_{\bullet} C_*\left(Y_{(k+1)}\right)
\end{array}$$ 
which fit into a direct limit map $f^G_* : \mathcal{N}_{\bullet} C_*(X ; G) \rightarrow \mathcal{N}_{\bullet} C_*(Y ; G)$. Of course, the operation which to $f$ associates $f^G_*$ is functorial as well.
\end{proof}

The additivity property of the Nash constructible filtration (Theorem 3.6 of \cite{MCP}) induces the additivity property of the equivariant Nash constructible filtration with respect to equivariant closed inclusions of $G$-$\mathcal{AS}$-sets :

\begin{theo} \label{addequivnashfil} Any equivariant closed inclusion of $G$-$\mathcal{AS}$-sets $Y \subset X$ induces a short exact sequence of filtered complexes
$$0 \rightarrow \mathcal{N}_{\bullet} C_*(Y ; G) \rightarrow \mathcal{N}_{\bullet} C_*(X ; G) \rightarrow \mathcal{N}_{\bullet} C_*(X \setminus Y ; G) \rightarrow 0.$$
\end{theo}

\begin{proof} Let $k \in \mathbb{N}$. The equivariant closed inclusion $Y \subset X$ induces an equivariant closed inclusion $Y \times V_{N}(\mathbb{R}^k) \subset X \times V_{N}(\mathbb{R}^k)$. We then apply the geometric quotient map $\pi_k$ to obtain the closed inclusion of $\mathcal{AS}$-sets $Y_{(k)} \subset X_{(k)}$ (recall lemma \ref{quotientinabigger}, and recall that $\pi_k$ is a closed map by lemma \ref{quotientproper}).

We then use the additivity of the Nash constructible filtration (Theorem 3.6 (2) of \cite{MCP}) to induce the short exact sequence of filtered complexes
$$0 \rightarrow \mathcal{N}_{\bullet} C_*(Y_{(k)}) \rightarrow \mathcal{N}_{\bullet} C_*(X_{(k)}) \rightarrow \mathcal{N}_{\bullet} C_*((X \setminus Y)_{(k)}) \rightarrow 0$$
(we have $X_{(k)} \setminus Y_{(k)} = (X \setminus Y)_{(k)}$). We conclude by taking the direct limit : the direct limit is an exact functor on modules.
\end{proof}

\begin{rem} By a diagram chasing, we also have short exact sequences 
$$0 \rightarrow \frac{\mathcal{N}_{p} C_*(Y ; G)}{\mathcal{N}_{p-1} C_*(Y ; G)} \rightarrow \frac{\mathcal{N}_{p} C_*(X ; G)}{\mathcal{N}_{p-1} C_*(X ; G)} \rightarrow \frac{\mathcal{N}_{p} C_*(X \setminus Y ; G)}{\mathcal{N}_{p-1} C_*(X \setminus Y ; G)} \rightarrow 0,$$
$p \leq 0$, of graded complexes.
\end{rem}

The acyclicity of the Nash constructible filtration induces the acyclicity of the equivariant Nash constructible filtration as well :

\begin{cor} \label{acyequivnashfil} Let
\begin{equation} \label{equivacyclicsquare}
\begin{array}{ccc} 
\widetilde{Y} & \stackrel{}{\longrightarrow} & \widetilde{X} \\
 ~ \Big\downarrow {} & & ~ \Big\downarrow {\scriptstyle s}\\
Y & \stackrel{i}{\longrightarrow} & X
\end{array}
\end{equation}
be an acyclic square of $G$-$\mathcal{AS}$-sets, i.e. a commutative diagram of $G$-$\mathcal{AS}$-sets and equivariant proper continuous map such that $i$ is an equivariant closed inclusion, $\widetilde{X} = s^{-1}(X)$ and the restriction $s : \widetilde{X} \setminus \widetilde{Y} \rightarrow X \setminus Y$ is an equivariant homeomorphism. It induces a short exact sequence of filtered complexes
$$0 \rightarrow \mathcal{N}_{\bullet} C_*(\widetilde{Y} ; G) \rightarrow \mathcal{N}_{\bullet} C_*(Y ; G) \oplus \mathcal{N}_{\bullet} C_*(\widetilde{X} ; G) \rightarrow \mathcal{N}_{\bullet} C_*(X ; G) \rightarrow 0.$$
\end{cor}

\begin{proof} The above acyclic square induces, by additivity of the equivariant Nash constructible filtration, the following commutative diagram of short exact sequences
$$\begin{array}{ccccccccc}
0 & \rightarrow & \mathcal{N}_{\bullet} C_*(\widetilde{Y} ; G) & \rightarrow & \mathcal{N}_{\bullet} C_*(\widetilde{X} ; G) & \rightarrow & \mathcal{N}_{\bullet} C_*(\widetilde{X} \setminus \widetilde{Y} ; G) & \rightarrow & 0 \\
   &                   & \Big\downarrow {}&    & \Big\downarrow  {}     &     {}                 & \Big\downarrow {\scriptstyle \cong} &&  \\
0 & \rightarrow & \mathcal{N}_{\bullet} C_*(Y ; G) & \rightarrow & \mathcal{N}_{\bullet} C_*(X ; G) & \rightarrow & \mathcal{N}_{\bullet} C_*(X \setminus Y ; G) & \rightarrow & 0
\end{array}$$
Now, the above short exact sequence of the statement follows from a diagram chasing.
\end{proof}

\begin{rem} By a diagram chasing argument, we have short exact sequences 
$$0 \rightarrow \frac{\mathcal{N}_{p} C_*(\widetilde{Y} ; G)}{\mathcal{N}_{p-1} C_*(\widetilde{Y} ; G)} \rightarrow \frac{\mathcal{N}_{p} C_*(Y ; G)}{\mathcal{N}_{p-1} C_*(Y ; G)} \oplus \frac{\mathcal{N}_{p} C_*(\widetilde{X} ; G)}{\mathcal{N}_{p-1} C_*(\widetilde{X} ; G)} \rightarrow \frac{\mathcal{N}_{p} C_*(X ; G)}{\mathcal{N}_{p-1} C_*(X ; G)} \rightarrow 0.$$
\end{rem}

\subsection{The induced equivariant weight spectral sequence} \label{subsechomequivweightspecseq}

For $X$ a $G$-$\mathcal{AS}$-set, the filtered complex $\mathcal{N}_{\bullet} C_*(X ; G)$ induces a spectral sequence that we denote by $E^*_{*,*} (X ; G)$. Since the equivariant Nash filtration filtration is bounded, the induced spectral sequence $E^*_{*,*} (X ; G)$ converges to $H_*(C_*(X ; G)) = H_*(X ; G)$.

\begin{de} \label{defhomequivweightfil} We call $E^*_{*,*} (X ; G)$ the equivariant weight spectral sequence of $X$ and we call the induced filtration
$$0 = \mathcal{N}_{-q-1} H_q(X ; G) \subset \mathcal{N}_{-q} H_q(X ; G) \subset \cdots \subset \mathcal{N}_{-1} H_q(X ; G) \subset \mathcal{N}_{0} H_q(X ; G) = H_q(X ; G)$$
on the equivariant homology of $X$ with closed supports, the equivariant weight filtration of $X$.
\end{de}

\begin{rem} These equivariant weight spectral sequence and filtration are different from the ones obtained in \cite{Pri-EWF}.
\end{rem}

Moreover, $E^*_{*,*} (X ; G)$ is the direct limit spectral sequence of the inductive system of spectral sequences $E^*_{*,*}(X_{(k)})$, $k \in \mathbb{N}$, where $E^*_{*,*}( \, \cdot \, )$ is the weight spectral sequence induced by the Nash constructible filtration (see \cite{MCP}). Indeed, the direct limit of the inductive system of spectral sequences induced by an inductive system of filtered complexes is the spectral sequence induced by the direct limit of the filtered complexes (use the definition of the direct limit and the exactness of the direct limit functor on modules).  

In particular, for all, $r$, $p$, $q \in \mathbb{Z}$,
$$E^r_{p,q} (X ; G) = \varinjlim_{k \in \mathbb{N}} E^r_{p,q}\left(X_{(k)}\right).$$

As in \cite{MCP} subsection 1C, we reindex the spectral sequences $E^r_{p,q}$ into spectral sequences $\widetilde{E}^{r'}_{p',q'}$, well-defined from $r' = 1$, by setting $r' = r+1$, $p' = 2 p + q$, $q' = -p$. Since, for each $k \in \mathbb{N}$, the non-zero terms of $\widetilde{E}^{r}_{p,q}\left(X_{(k)}\right)$ lie in the closed triangle with vertices $(0,0)$, $(0, d_k)$ and $(d_k, 0)$, where $d_k$ is the dimension of $X_{(k)}$, the spectral sequence $\widetilde{E}^r_{p,q} (X ; G)$ is a first quadrant spectral sequence.

We will show in theorem \ref{equivspecseqbound} below that it is right-bounded as well. First, let us mention how the additivity and acyclicity properties of the equivariant Nash constructible filtration translates on the induced spectral sequence :

\begin{lem} \label{addacyclequivweightspecseq} \begin{enumerate}
	\item Let $Y \subset X$ be an equivariant closed inclusion of $G$-$\mathcal{AS}$-sets. For any $q \in \mathbb{N}$, it induces a long exact sequence
$$ \cdots \rightarrow \widetilde{E}^2_{p,q} (Y ; G) \rightarrow \widetilde{E}^2_{p,q} (X ; G) \rightarrow \widetilde{E}^2_{p,q} (X \setminus Y ; G) \rightarrow \widetilde{E}^2_{p-1,q} (Y ; G) \rightarrow \cdots$$
 on the $q$\textsuperscript{th} line of the second page of the (reindexed) equivariant weight spectral sequence.
	\item Consider an acyclic square (\ref{equivacyclicsquare}) of $G$-$\mathcal{AS}$-sets. For any $q \in \mathbb{N}$, it induces a long exact sequence
$$ \cdots \rightarrow \widetilde{E}^2_{p,q} (\widetilde{Y} ; G) \rightarrow \widetilde{E}^2_{p,q} (Y ; G) \oplus \widetilde{E}^2_{p,q} (\widetilde{X} ; G)  \rightarrow \widetilde{E}^2_{p,q} (X ; G) \rightarrow \widetilde{E}^2_{p-1,q} (\widetilde{Y} ; G) \rightarrow \cdots$$
 on the $q$\textsuperscript{th} line of the second page of the equivariant weight spectral sequence.
\end{enumerate}
\end{lem}

\begin{proof} These long exact sequences are induced by the short exact sequences of additivity and acyclicity of the equivariant Nash constructible filtration just as in \cite{MCP} section 1C.
\end{proof}

\begin{theo} \label{equivspecseqbound} Let $d$ be the dimension of $X$. If $r \geq 2$ and $\widetilde{E}^r_{p,q} (X ; G) \neq 0$ then $0 \leq p \leq \dim X$. 
\end{theo}

\begin{proof} We use an induction on the dimension of $X$ (see also \cite{GF}, proof of Proposition 3.10 for instance).
\\

First, suppose that $X$ is a compact and nonsingular $\mathcal{AS}$-set, that is $X$ does not intersect the set of singular points of its Zariski closure (in $\mathbb{P}(\mathbb{R})$). Then, if $k \in \mathbb{N}$, the $\mathcal{AS}$-set $X_{(k)} = \pi_k(X \times V_{N}(\mathbb{R}^k))$ is also compact ($\pi_k$ is a continuous map) and nonsingular (proposition \ref{quotientnonsingiffree}). As a consequence, $X_{(k)}$ is a compact Nash submanifold of an affine space (see for instance Proposition 3.3.11 of \cite{BCR}) and its weight spectral sequence, for $r \geq 2$, is then concentrated in column $p = 0$ (Theorem 3.7 of \cite{MCP}), i.e. $E^r_{p,q}\left(X_{(k)}\right) = E^2_{p,q}\left(X_{(k)}\right) = 0$ if $p \neq 0$. Hence the same property for the direct limit : $E^r_{p,q}\left(X ; G \right) = E^2_{p,q}\left(X ; G \right)= 0$ if $p \neq 0$.

This case apply in particular when $X$ is zero-dimensional, that is when $X$ is a finite union of points.
\\

Now, suppose that $X$ is non-compact and nonsingular. There exists a compact and nonsingular $G$-$\mathcal{AS}$-set $\overline{X}$ such that $X$ can be equivariantly biregularly embedded in $\overline{X}$ and $\dim(\overline{X} \setminus X) < \dim X$. Indeed, consider the $\mathcal{AS}$-closure $\overline{X}^{\mathcal{AS}}$ of $X$ as well as the (projective) Zariski closure $\overline{X}^{\mathcal{Z}}$ of $X$, and consider an equivariant resolution of singularities $s$ of $\overline{X}^{\mathcal{Z}}$ (which exists by \cite{BM}). Since $X$ is away from the singularities of $\overline{X}^{\mathcal{Z}}$, $s^{-1}(X)$ is equivariantly biregularly isomorphic to $X$ and $s^{-1}\left(\overline{X}^{\mathcal{AS}}\right)$ is a compact ($s$ is proper) and nonsingular $G$-$\mathcal{AS}$-set (the inverse image of an $\mathcal{AS}$-set by a map with $\mathcal{AS}$-graph is an $\mathcal{AS}$-set) such that $\dim\left(s^{-1}\left(\overline{X}^{\mathcal{AS}}\right) \setminus s^{-1}(X)\right) < \dim X$. 
    
We then denote $Y := \overline{X} \setminus X$. The $\mathcal{AS}$-closure of $Y$ is a $G$-$\mathcal{AS}$-set as well and $\overline{Y}^{\mathcal{AS}} \cap X \subset X$ is an equivariant closed inclusion. On the other hand, $X \setminus \left(\overline{Y}^{\mathcal{AS}} \cap X\right) = \overline{X} \setminus \overline{Y}^{\mathcal{AS}} \subset \overline{X}$ is an equivariant open inclusion. Consider the long exact sequence of additivity of lemma \ref{addacyclequivweightspecseq} associated to the equivariant open inclusion $X \setminus \left(\overline{Y}^{\mathcal{AS}} \cap X\right) \subset \overline{X}$ :
$$ \cdots \rightarrow \widetilde{E}^2_{p,q} \left(\overline{Y}^{\mathcal{AS}} ; G\right) \rightarrow \widetilde{E}^2_{p,q} \left(\overline{X} ; G\right) \rightarrow \widetilde{E}^2_{p,q} \left(X \setminus \left(\overline{Y}^{\mathcal{AS}} \cap X\right) ; G\right) \rightarrow \widetilde{E}^2_{p-1,q} \left(\overline{Y}^{\mathcal{AS}} ; G\right) \rightarrow \cdots.$$
If $p \neq 0$, $\widetilde{E}^2_{p,q} \left(\overline{X} ; G\right) = 0$ by the previous case and, by the induction hypothesis, $\widetilde{E}^2_{p,q} \left(\overline{Y}^{\mathcal{AS}} ; G\right) = 0$ if $p \geq d$ (recall that $\dim \overline{Y}^{\mathcal{AS}} = \dim Y < \dim X$). Consequently, if $p > d$,  
$$\widetilde{E}^2_{p,q} \left(X \setminus \left(\overline{Y}^{\mathcal{AS}} \cap X\right) ; G\right) = 0.$$
Finally, consider the long exact sequence of additivity associated to the equivariant closed inclusion $\overline{Y}^{\mathcal{AS}} \cap X \subset X$ : 
$$ \cdots \rightarrow \widetilde{E}^2_{p,q} \left(\overline{Y}^{\mathcal{AS}} \cap X ; G\right) \rightarrow \widetilde{E}^2_{p,q} \left(X ; G\right) \rightarrow \widetilde{E}^2_{p,q} \left(X \setminus \left(\overline{Y}^{\mathcal{AS}} \cap X\right) ; G\right) \rightarrow \widetilde{E}^2_{p-1,q} \left(\overline{Y}^{\mathcal{AS}} \cap X ; G\right) \rightarrow \cdots$$
We use again the induction hypothesis ($\dim \overline{Y}^{\mathcal{AS}} \cap X < \dim X$) to deduce that $\widetilde{E}^2_{p,q} \left(X ; G\right) = 0$ if $p > d$.
\\

We conclude the proof with the general case : if $X$ is singular, the singular points of $\overline{X}^{\mathcal{Z}}$ included in $X$ form a closed $\mathcal{AS}$-subset of $X$ which is globally stabilized under the (biregular) action of $G$ and of dimension stricly smaller than $\dim X$ (see for instance Proposition 3.3.14 of \cite{BCR}). We can then use the previous cases along with the induction hypothesis to conclude. 
\end{proof}

As a consequence, the long exact sequences of additivity and acyclicity of lemma \ref{addacyclequivweightspecseq} are actually finite long exact sequences. This allows us to define the following invariants :

\subsection{The equivariant virtual Betti numbers}

\begin{theo} \label{equivvbn} Let $X$ be a $G$-$\mathcal{AS}$-set and let $q \in \mathbb{N}$. We denote
$$\beta_q(X ; G) := \sum_{p \in \mathbb{N}} (-1)^p \dim_{\mathbb{Z}_2} \widetilde{E}^2_{p,q} \left(X ; G\right)$$ 
the $q$\textsuperscript{th} equivariant virtual Betti number of $X$.
\\

The $q$\textsuperscript{th} equivariant virtual Betti number $\beta_q( \, \cdot \,; G)$ has values in $\mathbb{Z}$ and is 
\begin{enumerate}
	\item an invariant of $G$-$\mathcal{AS}$-sets with respect to equivariant homeomorphisms with $\mathcal{AS}$-graph, 
	\item additive with respect to equivariant closed inclusions of $G$-$\mathcal{AS}$-sets, i.e., if $Y \subset X$ is an equivariant closed inclusion, $\beta_q(X ; G) = \beta_q( Y  ; G) + \beta_q( X \setminus Y ; G)$,
	\item coincides with the dimension (over $\mathbb{Z}_2$) of the $q$\textsuperscript{th} equivariant homology group on compact nonsingular $G$-$\mathcal{AS}$-sets. 
\end{enumerate}

Moreover, the $q$\textsuperscript{th} equivariant virtual Betti number is unique with these properties.
\end{theo}

\begin{proof} First, notice that $\beta_q(X ; G)$ is well-defined since the sum over $p$ is finite by theorem \ref{equivspecseqbound}.

The $q$\textsuperscript{th} equivariant virtual Betti number is invariant with respect to equivariant homeomorphisms with $\mathcal{AS}$-graph because so is the equivariant Nash constructible filtration and the induced equivariant weight spectral sequence. It is additive because of the long exact sequence of additivity of lemma \ref{addacyclequivweightspecseq}, which is finite for a given equivariant closed inclusion by theorem \ref{equivspecseqbound}.

If $X$ is compact nonsingular, the equivariant weight spectral sequence of $X$ converges at $\widetilde{E}^2 \left(X ; G\right)$ and is concentrated in the column $p = 0$ (see the proof of theorem \ref{equivspecseqbound}). Since the equivariant weight spectral sequence of $X$ converges to the equivariant homology of $X$ with closed supports, we have
$$\beta_q(X ; G) = \dim_{\mathbb{Z}_2} \widetilde{E}^2_{0,q} = \dim_{\mathbb{Z}_2} H_q(X ; G) = \dim_{\mathbb{Z}_2} H^G_q(X)$$ 
($X$ is compact : see lemma \ref{coincideequivhomsingclosed}). 
\\

We finally show the uniqueness of the $q$\textsuperscript{th} equivariant virtual Betti number. Consider an application $B_q( \, \cdot \, ; G)$ with the same above properties 1), 2) and 3) as $\beta_q( \, \cdot \,; G)$. We prove that $B_q(X ; G) = \beta_q(X ; G)$ for any $G$-$\mathcal{AS}$-set $X$, proceeding by induction on the dimension.

Just as in the proof of theorem \ref{equivspecseqbound}, suppose first that $X$ is compact and nonsingular. Then
$$B_q(X ; G) = \dim_{\mathbb{Z}_2} H_q(X ; G) = \beta_q(X ; G).$$

Secondly, suppose that $X$ is nonsingular and non-compact and consider an equivariant nonsingular compactification $\overline{X}$ of $X$ such that $\dim(\overline{X} \setminus X) < \dim X$. Denote $Y := \overline{X} \setminus X$. Then $\overline{Y}^{\mathcal{AS}} \cap X \subset X$ is an equivariant closed inclusion, so
$$B_q(X ; G) = B_q\left(\overline{Y}^{\mathcal{AS}} \cap X ; G\right) + B_q\left( X \setminus \overline{Y}^{\mathcal{AS}} ; G \right),$$
and $X \setminus \overline{Y}^{\mathcal{AS}} = \overline{X} \setminus \overline{Y}^{\mathcal{AS}} \subset \overline{X}$ is an equivariant open inclusion, so that
$$B_q\left( X \setminus \overline{Y}^{\mathcal{AS}} ; G \right) = B_q\left( \overline{X} ; G\right) - B_q\left(\overline{Y}^{\mathcal{AS}}; G \right).$$ 
As a consequence, $B_q(X ; G) = \beta_q(X ; G)$ thanks to the induction hypothesis and the previous case ($\overline{X}$ is compact nonsingular).

The final step consists in considering the closed subset of singular points of $X$ and to use again the additivity of $B_q(\, \cdot \, ;  G)$ and $\beta_q(\, \cdot \, ; G)$, the induction hypothesis and the previous cases.
\end{proof}

\begin{rem} 
\begin{itemize}
	\item In the proof of uniqueness, we used the invariance of $\beta_q(\, \cdot \,; G)$ with respect to equivariant biregular isomorphisms in order to replace $X$ by a subset of a compact nonsingular $G$-$\mathcal{AS}$ set with complement of strictly smaller dimension. Actually, we showed that any additive invariant of $G$-$\mathcal{AS}$-sets with respect to equivariant biregular isomorphisms, with values in $\mathbb{Z}$ and coinciding with the dimension of the $q$\textsuperscript{th} equivariant homology group on compact nonsingular $G$-$\mathcal{AS}$-sets, coincides with $\beta_q(\, \cdot \,; G)$ and is, in particular, invariant with respect to equivariant homeomorphisms with $\mathcal{AS}$-graph.

\item If $X$ is a real algebraic set equipped with an action of $G$ via biregular isomorphisms, to compute $\beta_q(X ; G)$, we can first consider an equivariant open compactification $\overline{X}$ of $X$ and then an equivariant resolution of singularities of $\overline{X}$. We can construct an equivariant open compactification of $X$ using a trick similar to lemma \ref{linearisationfinigroup} : consider an (a priori non-equivariant) open compactification $i : X \hookrightarrow \widetilde{X}$ of $X$ and construct the morphism
$$j : \begin{array}{ccc} 
		X & \rightarrow &  \widetilde{X} \times \cdots \times  \widetilde{X} \\
		x & \mapsto & \left(i(x), i\left(\alpha_{g_2^{-1}}(x)\right), \ldots, i\left(\alpha_{g_N^{-1}}(x)\right)\right)
		\end{array}$$
(if $G = \{g_1, \ldots, g_N\}$). The Zariski closure of $j(X)$ in $\widetilde{X} \times \cdots \times  \widetilde{X}$ is then an equivariant open compactification of $X$ (the action is given by the permutations induced by the product in $G$).
\end{itemize}
\end{rem}

Actually, we can relax the ``closed inclusion'' hypothesis :

\begin{cor} If $Y \subset X$ is any inclusion of $G$-$\mathcal{AS}$-sets, then
$$\beta_q(X ; G) = \beta_q(Y ; G) + \beta_q(X \setminus Y ; G)$$
for all $q \in \mathbb{N}$.
\end{cor}

\begin{proof} Let $q \in \mathbb{N}$. First, let us prove that, if $T$ be a $G$-$\mathcal{AS}$-set,
$$\beta_q(T ; G) = \beta_q\left(\overline{T}^{\mathcal{AS}}\right) - \beta_q\left(\overline{T}^{\mathcal{AS}} \setminus T ; G\right).$$
We proceed by induction on the dimension : the property is true for zero-dimensional $\mathcal{AS}$-sets, and suppose it to be true for $G$-$\mathcal{AS}$-sets of dimension $\leq d-1$. Let $T$ be a $d$-dimensional $G$-$\mathcal{AS}$-set and denote $S :=  \overline{T}^{\mathcal{AS}} \setminus T$ (we have $\dim S < \dim T$ : see \cite{KP} Proposition 3.3). We have
$$\beta_q(T ; G) = \beta_q\left(\overline{T}^{\mathcal{AS}} ; G\right) + \beta_q\left(\overline{S}^{\mathcal{AS}} \cap T ; G\right) - \beta_q\left(\overline{S}^{\mathcal{AS}} ; G\right)$$
(see the proof of theorem \ref{equivvbn}). But $\overline{S}^{\mathcal{AS}} \cap T \subset \overline{S}^{\mathcal{AS}}$ is an inclusion of $G$-$\mathcal{AS}$-set of dimension $< d$ so
$$\beta_q\left(\overline{S}^{\mathcal{AS}} ; G\right) - \beta_q\left(\overline{S}^{\mathcal{AS}} \cap T ; G\right)  = \beta_q\left(\overline{S}^{\mathcal{AS}} \setminus T; G\right) = \beta_q\left(\overline{T}^{\mathcal{AS}} \setminus T\right).$$

\vspace{0.5cm}

We will now show that
$$\beta_q(X ; G) = \beta_q(Y ; G) + \beta_q(X \setminus Y ; G)$$
for any inclusion $Y \subset X$ of $G$-$\mathcal{AS}$-sets, proceeding once again by induction on the dimension (the property is obviously true for zero-dimensional $\mathcal{AS}$-sets) : suppose the above equality to be true for any $G$-$\mathcal{AS}$-sets of dimension $\leq d-1$ and consider an inclusion $Y \subset X$ of $G$-$\mathcal{AS}$-sets of dimension $\leq d$.

If $X$ is compact, $\overline{Y}^{\mathcal{AS}} \subset X$ and $\overline{Y}^{\mathcal{AS}} \setminus Y \subset X \setminus Y$ are equivariant closed inclusions so that 
\begin{eqnarray*}
\beta_q(X \setminus Y ; G) & = & \beta_q\left(\overline{Y}^{\mathcal{AS}} \setminus Y ; G\right) + \beta_q\left(X \setminus \overline{Y}^{\mathcal{AS}} ; G\right) \\
					& = & \beta_q\left(\overline{Y}^{\mathcal{AS}} ; G\right) - \beta_q(Y ; G) + \beta_q(X ; G) - \beta_q\left( \overline{Y}^{\mathcal{AS}} ; G\right) \\
					& = & \beta_q(X ; G) - \beta_q(Y ; G)    
\end{eqnarray*}

If $X$ is not compact, denote $X_0 := \overline{X}^{\mathcal{AS}} \setminus X$ and consider the equivariant closed inclusion $\overline{X_0}^{\mathcal{AS}} \cap ( X \setminus Y) \subset X \setminus Y$. We have 
$$\beta_q(X \setminus Y ; G) = \beta_q\left(\overline{X_0}^{\mathcal{AS}} \cap (X \setminus Y) ; G\right) + \beta_q\left( (X \setminus Y) \setminus \overline{X_0}^{\mathcal{AS}} ; G \right).$$
Since $(X \setminus Y) \setminus \overline{X_0}^{\mathcal{AS}} = \left(\overline{X}^{\mathcal{AS}} \setminus Y\right) \setminus \overline{X_0}^{\mathcal{AS}} \subset \overline{X}^{\mathcal{AS}} \setminus Y$ is an equivariant open inclusion, we also have
$$\beta_q\left( (X \setminus Y) \setminus \overline{X_0}^{\mathcal{AS}} ; G \right) = \beta_q\left(\overline{X}^{\mathcal{AS}} \setminus Y ; G \right) - \beta_q\left( \overline{X_0}^{\mathcal{AS}} \setminus Y ; G\right).$$
Finally, $\overline{X_0}^{\mathcal{AS}} \cap (X \setminus Y)  \subset \overline{X_0}^{\mathcal{AS}} \setminus Y$ is an equivariant inclusion in dimension $<d$ so
$$\beta_q\left( \overline{X_0}^{\mathcal{AS}} \setminus Y ; G\right) - \beta_q\left(\overline{X_0}^{\mathcal{AS}} \cap (X \setminus Y) ; G\right) = \beta_q\left(\overline{X}^{\mathcal{AS}} \setminus X\right)$$
and
\begin{eqnarray*}
\beta_q(X \setminus Y ; G) & = & \beta_q\left(\overline{X}^{\mathcal{AS}} \setminus Y ; G \right) -  \beta_q\left(\overline{X}^{\mathcal{AS}} \setminus X ; G\right) \\
					 & = & \beta_q\left(\overline{X}^{\mathcal{AS}} ; G \right) - \beta_q(Y ; G) - \left(\beta_q\left(\overline{X}^{\mathcal{AS}} ; G \right) - \beta_q(X ; G) \right) \\
					 & = & \beta_q(X ; G) - \beta_q(Y ; G).
\end{eqnarray*}
\end{proof}

Let us then give the following definition :

\begin{de} Let $e$ be an application from the category of $G$-$\mathcal{AS}$-sets and equivariant continuous maps with $\mathcal{AS}$-graph to a ring $A$. We say that $e$ is an additive invariant of $G$-$\mathcal{AS}$-sets if
\begin{itemize}
	\item whenever $f : X  \rightarrow X'$ is an equivariant homeomorphism with $\mathcal{AS}$-graph, then
	$$e(X) = e(X'),$$
	\item for any equivariant inclusion $Y \subset X$, we have
	$$e(X) = e(Y) + e(X \setminus Y).$$ 
\end{itemize}
\end{de}

The equivariant virtual Betti numbers are additive invariants of $G$-$\mathcal{AS}$-sets. We define another one from them :

\begin{de} \label{equivvps} Let $X$ a $G$-$\mathcal{AS}$-set. We set
$$\beta(X ; G) := \sum_{q \in \mathbb{N}} \beta_q(X ; G) u^q \in \mathbb{Z}[[u]].$$
and we call this formal power series the equivariant virtual Poincar\'e series of $X$.

The equivariant virtual Poincar\'e series is an additive invariant of $G$-$\mathcal{AS}$-sets, which coincides on compact nonsingular $G$-$\mathcal{AS}$-sets with the generating function of the dimensions of the equivariant homology groups (we denote this generating function by $b(\, \cdot \, ; G)$).
\end{de}

\begin{rem} \begin{itemize}
	\item The equivariant virtual Betti numbers and the equivariant virtual Poincar\'e series defined above are different from the ones in \cite{GF} : they are not induced by the same equivariant homology. In particular, the above equivariant virtual Poincar\'e series of definition \ref{equivvps} does not encode the dimension (see below example \ref{exequivpoincser} (2)).
	\item Let $q \in \mathbb{N}$. The application which associates to a $G$-$\mathcal{AS}$-set $X$ the $q$\textsuperscript{th} virtual Betti number (see \cite{MCP-VB}, \cite{GF-MI} and also \cite{MCP}) of its fixed points set $\beta_q(X^G)$ is an additive invariant of $G$-$\mathcal{AS}$-set with values in $\mathbb{Z}$. 
	
	The virtual Poincar\'e polynomial of the fixed points set is also an additive invariant of $G$-$\mathcal{AS}$-set, with values in $\mathbb{Z}[u]$.

	\item If $G = \{e\}$, $\beta( \, \cdot \, ; G)$ is the virtual Poincar\'e polynomial of \cite{GF-MI}, since $H_*(X ; G) = H_*(X)$ if $X$ is a compact (nonsingular) $\mathcal{AS}$-set (see remark \ref{remequivsinghom} (5) and lemma \ref{coincideequivhomsingclosed}).
\end{itemize}
\end{rem}

\begin{ex} \label{exequivpoincser} \begin{enumerate}
	\item Consider the $d$-dimensional affine space $\mathbb{A} := \mathbb{R}^d$ equipped with any orthogonal action of a finite group $G$. In order to compute the equivariant virtual Poincar\'e series of $\mathbb{A}$, we consider the radial projection of $\mathbb{A}$ into the $d$-dimensional sphere $\mathbb{S}$ with center $(0, \ldots, 0, \frac{1}{2})$ and radius $\frac{1}{2}$ of $\mathbb{R}^{d+1}$ : see for instance the proof of Proposition 3.5.12 of \cite{BCR}. 
	
	If we naturally extend the orthogonal action of $G$ into an orthogonal action on $\mathbb{R}^{d+1}$ (take the diagonal action fixing the last coordinate), this (bi)regular embedding is equivariant (because the action preserves the euclidean norm). Denote by $p$ the point $(0, \ldots, 0, 1)$ of $\mathbb{R}^{d+1}$. We then have
$$\beta(\mathbb{A} ; G) = \beta(\mathbb{S} ; G) - \beta(p ; G) =  b(\mathbb{S} ; G) - b(p ; G),$$   
since $\mathbb{S}$ and $p$ are compact and nonsingular.

We have $b(p ; G) =  \sum_{q \in \mathbb{N}} H_q(G, \mathbb{Z}_2) u^q$ (remark \ref{remequivsinghom} (5)), and consider the $G$-$CW$-structure on $\mathbb{S}$ consisting in the $G$-invariant $d$-cell $\mathbb{A}$ and the $G$-invariant $0$-cell $p$. Since all the cells are globally invariant under $G$, the action on the cellular complex $C^{cell}_*(\mathbb{S})$ is trivial and $$H_*(\mathbb{S} ; G) = H_*(G,C^{cell}_*(\mathbb{S})) = H_*(\mathbb{S}) \otimes_{\mathbb{Z}_2} H_*(G , \mathbb{Z}_2)$$ 
(see \cite{Brown} VII-5 (5.4)), so that
$$b(\mathbb{S} ; G) = b(\mathbb{S}) \left( \sum_{q \in \mathbb{N}} H_q(G, \mathbb{Z}_2) u^q \right) = (1+ u^d) \left( \sum_{q \in \mathbb{N}} H_q(G, \mathbb{Z}_2) u^q \right)$$  
($b(\, \cdot \,)$ denotes the Poincar\'e polynomial). Finally
$$\beta(\mathbb{A} ; G) = u^d \left( \sum_{q \in \mathbb{N}} H_q(G, \mathbb{Z}_2) u^q \right) = u^d b(p ; G).$$

	\item We consider two different actions of $G := \mathbb{Z}/2 \mathbb{Z}$ on the hyperbola $X := \{x y = 1\}$ of $\mathbb{R}^2$. 
	
	First, consider the action given by the involution $\sigma : (x,y) \mapsto (-x, -y)$. Consider the projective Zariski closure $\overline{X} := \{XY = Z^2\}$ of $X$ in $\mathbb{P}^2(\mathbb{R})$, equipped with the involution $\overline{\sigma} : (X : Y : Z) \mapsto (-X : -Y :Z) = (X : Y : -Z)$. Then $X \subset \overline{X}$ is an equivariant inclusion (compactification) of $G$-$\mathcal{AS}$-sets and $\overline{X} \setminus X = \{p,q \}$, where $p$ and $q$ are the points of respective homogeneous coordinates $[1 : 0 : 0]$ and $[0 : 1 : 0]$. Notice that $p$ and $q$ are fixed by the action of $G$ on $\overline{X}$ and that $\overline{X}$ is a nonsingular compact $G$-$\mathcal{AS}$-set equivariantly homeomorphic to a circle equipped with a continuous action of $G$ with two fixed points. As a consequence,
$$\beta(X ; G) = \beta(\overline{X} ; G) - \beta(\{p,q\} ; G) = (1+ u) b(p ; G) - 2 b(p ; G) = (u-1) \sum_{q \in \mathbb{N}} u^q = -1$$  
($H_k(G , \mathbb{Z}_2) = \mathbb{Z}_2$ if $k \geq 0$ for $G = \mathbb{Z}/2 \mathbb{Z}$ : see for instance \cite{GF}  Example 2.1).	
	
	If now we consider the action of $G$ given by the involution $\sigma : (x,y) \mapsto (y, x)$ (notice that the points of coordinates $(1,1)$ and $(-1,-1)$ are fixed by the $\sigma$), we equip $\overline{X}$ with the involution $\overline{\sigma} : (X : Y : Z) \mapsto (Y : X :Z)$, for which the two points $p$ and $q$ are exchanged. Consequently, $\beta(\{p,q\} ; G) = b(\{p,q\}/G) = b(p)=1$ and
$$\beta(X ; G) = \beta(\overline{X} ; G) - \beta(\{p,q\} ; G) = 	(1+u )\sum_{q \in \mathbb{N}} u^q - 1 = \sum_{q \geq 1} 2 u^q.$$
In particular, we can see through this example that the equivariant virtual Poincar\'e series does not encode dimension.

\item Let $k$ and $l$ be odd integers and let $X$ be the real algebraic set $\{y^{2l} = x^{2 k}(1- x^{2k})$ of $\mathbb{R}^2$. We will consider the actions of $G := \mathbb{Z}/2 \mathbb{Z}$ on $X$ given by the involutions $\sigma_1 : (x , y) \mapsto (-x,y)$, $\sigma_2 : (x,y) \mapsto (x,-y)$ and $\sigma_3 : (x,y) \mapsto (-x,-y)$. First, notice that the map $(x,y) \mapsto (x^k , y^l)$ induces an equivariant homeomorphism with $\mathcal{AS}$-graph (actually algebraic graph) between the algebraic set $X' := \{y^{2} = x^{2}(1- x^{2})$ of $\mathbb{R}^2$ and $X$, so that
$$\beta(X ; G) = \beta(\widetilde{X} ; G).$$

As in \cite{GF} Example 4.6, we consider an equivariant resolution of $X'$ to compute the equivariant virtual Poincar\'e series of $X$. With the values of our equivariant virtual Poincar\'e series on the circle and points, we obtain
$$\beta(X ; G) = \begin{cases}
				1 + 2 u + \sum_{q \geq 2} u^q & \mbox{ if $G = \{id, \sigma_1\}$,} \\
				1 + \sum_{q \geq 1} 2 u^q & \mbox{ if $G = \{id, \sigma_2\}$,} \\
				\sum_{q \geq 1} u^q & \mbox{ if $G = \{id, \sigma_3\}$.}
			\end{cases}$$
In particular, we see that the actions of the involutions $\sigma_1$ and $\sigma_3$ are different up to equivariant homeomorphism with $\mathcal{AS}$-graph, which could not be proven using the equivariant virtual Poincar\'e series of $\cite{GF}$ or the virtual Poincar\'e polynomial of the fixed points set.
\end{enumerate}
\end{ex}

\subsection{The dual equivariant Nash constructible filtration}

We deal with the cohomological counterpart of the previous paragraphs, adapting the construction of the dual geometric filtration of \cite{LP} section 4 to our equivariant $\mathcal{AS}$ context. First, remark that $C^q(X ; G) = \Hom_{\mathbb{Z}_2} \left(C_q(X ; G) , \mathbb{Z}_2 \right)$, since
\begin{equation} \label{limhomcom}
\varprojlim_{k \in \mathbb{N}} \Hom_{\mathbb{Z}_2} \left(M_k, \mathbb{Z}_2 \right) = \Hom_{\mathbb{Z}_2} \left( \varinjlim_{k \in \mathbb{N}}\ M_k, \mathbb{Z}_2 \right)
\end{equation}
for any inductive system of $\mathbb{Z}_2$-vector spaces $M_k$, $k \in \mathbb{N}$.

\begin{de} \label{defdualequivnashfil} Let $X$ be a $G$-$\mathcal{AS}$-set. For $p \in \mathbb{Z}$, $q \in \mathbb{N}$, we set
$$\mathcal{N}^p C^q(X ; G) := \left\{\varphi \in C^q(X ; G)~|~\varphi \equiv 0 \mbox{ on } \mathcal{N}_{p-1} C_q(X ; G)\right\}.$$
This fits into a decreasing filtration $\mathcal{N}^{\bullet}$ on $C^*(X ; G)$ :
$$0 = \mathcal{N}^1 C^q(X ; G) \subset \mathcal{N}^0 C^q(X; G) \subset \cdots \subset \mathcal{N}^{-q+1} C^q(X ; G) \subset \mathcal{N}^{-q} C^q(X ; G),$$
that we call the dual equivariant Nash constructible filtration of $X$.
\end{de}

Since the filtration $\mathcal{N}^{\bullet}$ is bounded, it induces a spectral sequence $E^{*,*}_*(X ; G)$ which converges (\cite{McCleary} Theorem 2.6) to the cohomology of $C^*(X; G)$, that is the equivariant cohomology with closed supports $H^*(X ; G)$ of $X$. The induced filtration 
$$0 = \mathcal{N}^1 H^q(X ; G) \subset \mathcal{N}^0 H^q(X; G) \subset \cdots \subset \mathcal{N}^{-q+1} H^q(X ; G) \subset \mathcal{N}^{-q} H^q(X ; G)$$
on $H^*(X ; G)$ is called the cohomological equivariant weight filtration of $X$.

Just as in \cite{LP} section 4, we have natural isomorphisms 
\begin{equation} \label{naturalisodual}
\mathcal{N}^p C^q(X ; G) = \left(\frac{C_q(X;G)}{\mathcal{N}_{p-1} C_q(X ; G)}\right)^{\vee}
\end{equation}
(where ${\, \cdot \,}^{\vee}$ denotes the duality functor $\Hom_{\mathbb{Z}_2}( \, \cdot \, , \mathbb{Z}_2)$, which is exact), so that the cohomological equivariant weight spectral sequence is naturally dual to the (homological) equivariant weight spectral sequence :
$$E^{p,q}_r(X ; G) = \left( E_{p,q}^r(X ; G)\right)^{\vee}$$
for all $r \geq 0$, $p,q \in \mathbb{Z}$.

Moreover, the filtered cochain complex $\mathcal{N}^{\bullet} C^*\left(X ; G\right)$ is the inverse limit of the projective system $\mathcal{N}^{\bullet} C^*\left(X_{(k+1)}\right) \rightarrow \mathcal{N}^{\bullet} C^*\left(X_{(k)}\right)$ and the cohomological equivariant weight spectral sequence is the inverse limit of the induced cohomological weight spectral sequences (use the above natural isomorphisms (\ref{naturalisodual}), (\ref{limhomcom}) and the exactness of the direct limit and duality functors).
\\

\begin{rem} The above cohomological equivariant weight filtration and spectral sequence are different from the ones of \cite{Pri-CPEWF} section 3. \\
\end{rem}

As in \cite{LP} Lemma 4.2, the additivity and acyclicity short exact sequences of theorem \ref{addequivnashfil} and corollary \ref{acyequivnashfil} induces short exact sequences of additivity and acyclicity for the dual equivariant Nash constructible filtration. We deduce finite long exact sequences of additivity (and acyclicity) on the lines of the second page of the reindexed (take the same reindexation as in subsection \ref{subsechomequivweightspecseq}) cohomological equivariant weight spectral sequence. We can therefore recover the equivariant virtual Betti numbers (theorem \ref{equivvbn}) from the cohomological equivariant weight spectral sequence as well :

\begin{prop}
Let $X$ be a $G$-$\mathcal{AS}$-set and let $q \in \mathbb{N}$. We have
$$\beta_q(X ; G) = \sum_{p \in \mathbb{N}} (-1)^p \dim_{\mathbb{Z}_2} \widetilde{E}_2^{p,q} \left(X ; G\right).$$ 
\end{prop}

\begin{proof} Each line of $\widetilde{E}_2 \left(X ; G\right)$ is bounded because it is dual to its homological counterpart and because of theorem \ref{equivspecseqbound}. Furthermore, thanks to the long exact sequence of additivity on each line of $\widetilde{E}_2^{p,q} \left(X ; G\right)$, the right member is additive. It is also an invariant of $G$-$\mathcal{AS}$-sets because the dual equivariant Nash constructible filtration is a functor with respect to equivariant proper continuous maps with $\mathcal{AS}$-graph, since so is the homological one (use again the natural isomorphisms (\ref{naturalisodual})).

Finally, if $X$ is compact and nonsingular, we have
$$\widetilde{E}^2_{p,q} \left(X ; G\right) = \begin{cases} H_q(X; G) = H_q^G(X) & \mbox{ if $p = 0$,} \\ 0  & \mbox{ if $p \neq 0$,} \end{cases}$$
and
$$\widetilde{E}_2^{p,q} \left(X ; G\right) = \left(\widetilde{E}^2_{p,q} \left(X ; G\right)\right)^{\vee} =  \begin{cases} H^q(X; G) = H^q_G(X) & \mbox{ if $p = 0$,} \\ 0  & \mbox{ if $p \neq 0$,} \end{cases}$$
so that $\sum_{p \in \mathbb{N}} (-1)^p \dim_{\mathbb{Z}_2} \widetilde{E}_2^{p,q} \left(X ; G\right) = \dim_{\mathbb{Z}_2} H_q^G(X)$.

We conclude by the uniqueness of the $q$\textsuperscript{th} equivariant virtual Betti number with these properties.
\end{proof}

\begin{rem} If $X$ is compact nonsingular, we have a Poincar\'e duality isomorphism between $H^*(X ; G)$ and the equivariant homology considered in \cite{GF} (see \cite{VanHamel} III Theorem 4.2, \cite{GF} 2.3.5 and also \cite{Pri-CPEWF} Remark 4.26) so that
$$\beta(X ; G)(u) = u^d \beta^G(X)(u^{-1})$$  
where $d$ is the dimension of $X$ and $\beta^G$ is the equivariant virtual Poincar\'e series of \cite{GF}. 

However, this equality does not hold for general $G$-$\mathcal{AS}$-sets. Consider for instance the third example of \ref{exequivpoincser} : we have $\beta\left(X ; \{id, \sigma_1\}\right) \neq \beta\left(X ; \{id, \sigma_3\}\right)$ while $\beta^{\{id, \sigma_1\}}(X) = \beta^{\{id, \sigma_3\}}(X)$ (see \cite{GF} Example 4.6).
\end{rem}

\section{Properties of the equivariant virtual Poincar\'e series and applications}

In this final section, we show that the equivariant virtual Poincar\'e series of definition \ref{equivvps} has properties similar to the ones of the equivariant virtual Poincar\'e series of \cite{GF}, which allows it to be used to define helpful tools (namely zeta functions) for the classification of real analytic germs.
\\

All the story begins with the following property, similar to \cite{GF} Proposition 3.13 :

\begin{prop} \label{propintmotequiv} Let $X$ be a $G$-$\mathcal{AS}$-set and let $\mathbb{R}^d$ be an affine space equipped with any orthogonal action of $G$. Then
$$\beta(X \times \mathbb{R}^d ; G) = u^d \beta(X ; G)$$
(on the left-hand side, we consider the diagonal action of $G$ on $X \times \mathbb{R}^d$).
\end{prop}

Before proving the above result, we need to mention the following facts. First, remark that, if $G$ and $H$ are finite groups and $X$, resp. $Y$, are topological spaces on which $G$, resp. $H$, act via homeomorphisms, we have a K\"unneth-type formula
$$H_*^{G} (X) \otimes_{\mathbb{Z}_2} H_*^H(Y) \cong H_*^{G \times H}( X \times Y).$$
Indeed, if $E_G$, resp. $E_H$, is a contractible topological space equipped with a free action of $G$, resp. $H$, then $E_G \times E_H$ is a contractible space with a free action of $G \times H$ and $(X \times Y) \times_{G \times H} (E_G \times E_H)$ is naturally isomorphic to $(X \times_G E_G) \times (Y \times_H E_H)$, so that the usual K\"unneth formula for homology can be applied. 
\\

We will use this property together with the following one : suppose that $X$ and $Y$ are two $G$-$CW$-complexes such that the action $G$ globally stabilizes each cell of $Y$, then $H^G_*(X \times Y) = H^{G \times \{e\}}_*(X \times Y)$, where, on the left-hand side, $G$ acts diagonally on $X \times Y$ and, on the right-hand side, we make the trivial group act on $Y$. 

Indeed, compute the equivariant homology of $X \times Y$ as the cellular homology of the quotient of $(X \times Y) \times E_G$ by $G$. Take $E_G$ to be a $G$-$CW$-complex such that its cells are freely permuted by $G$ (\cite{Hatcher}  Example 1B.7)~: a cell of the quotient is then the orbit of freely permuted cells of $(X \times Y) \times E_G$ (which are the products of the cells of $X$, $Y$ and $E_G$) under the action of $G$. Since the actions of $G$ and $G \times \{e\}$ on the cells of $(X \times Y) \times E_G$ have the same orbits, we get the result.

\begin{proof}[Proof of proposition \ref{propintmotequiv}] The proof is analog to the proof of Proposition 3.13 of \cite{GF} : first, suppose that $X$ is a compact and nonsingular $G$-$\mathcal{AS}$-set and equivariantly compactify $\mathbb{R}^d$ into the $d$-dimensional sphere $\mathbb{S}^d$ as in above example \ref{exequivpoincser} (1) : $X \times \mathbb{R}^d \hookrightarrow X \times \mathbb{S}^d$ is then an equivariant compactification of $X \times \mathbb{R}^d$. 

Consider a $G$-$CW$-structure on $X$ (\cite{ParkSuh}) and the $G$-$CW$-structure on $\mathbb{S}^d$ consisting in the $G$-invariant $d$-cell $\mathbb{R}^d$ and the $G$-invariant $0$-cell $\{p\} := \mathbb{S}^d \setminus \mathbb{R}^d$. Since $X$ and $\mathbb{S}^d$ are compact and nonsingular, we have
$$\beta(X \times \mathbb{S}^d ; G) = b(X \times \mathbb{S}^d ; G) = b(X \times \mathbb{S}^d ; G \times \{e\}) = b(X ; G) \cdot b(\mathbb{S}^d ; \{e\}) = b(X ; G) \cdot b(\mathbb{S}^d) = (1+u^d) b(X ; G), $$ 
and, by additivity of the equivariant virtual Poincar\'e series,
$$\beta(X \times \mathbb{R}^d ; G) = \beta(X \times \mathbb{S}^d ; G) - \beta(X \times \{p\} ; G) = (1+u^d) b(X ; G) - b(X ; G) = u^d b(X ; G) = u^d \beta(X ; G).$$

The rest of the proof proceeds just as in \cite{GF}, using an induction on the dimension of 
$X$ and the additivity of the equivariant virtual Poincar\'e series $\beta(\, \cdot \, ; G)$ (see also the proofs of theorems \ref{equivspecseqbound} and \ref{equivvbn}).
\end{proof}

\begin{rem} The equality is true as soon as the affine space $\mathbb{R}^d$ can be equivariantly compactified into a $d$-dimensional sphere.
\end{rem}

Thanks to proposition \ref{propintmotequiv}, the equivariant virtual Poincar\'e series could be used to define invariants, in terms of (motivic) zeta functions, of some equivalence relation of equivariant Nash germs, namely equivariant blow-Nash equivalence (see \cite{GF-ZF}) or equivariant arc-analytic equivalence (see \cite{JB-AAE}), just as as in \cite{Pri-EBN}. Indeed, this is this key property of the equivariant virtual Poincar\'e series of \cite{GF}, together with its additivity, which allow to prove Propositions 3.14 and 3.17 of \cite{Pri-EBN}.

This could be applied to study the classification of simple Nash germs invariant under the involution changing the sign of the first coordinate, as in \cite{Pri-ESNG}. 
\\

We also state the analogs of Proposition 3.14 and 3.15 of \cite{GF} for our equivariant virtual Poincar\'e series :

\begin{prop} \label{equivvpoinctrivfree} Let $X$ be a $G$-$\mathcal{AS}$-set.
\begin{enumerate}
	\item If the action of $G$ on $X$ is trivial, then $\beta(X ; G) = \beta(X) \left( \sum_{q \in \mathbb{N}} H_q(G, \mathbb{Z}_2) u^q \right)$.
	\item If $X$ is a free $G$-$\mathcal{AS}$-set, then the quotient $X/G$ is well-defined as an $\mathcal{AS}$-set (corollary \ref{quotientas} and remark \ref{remdifemb}) and $\beta(X ; G) = \beta(X/G)$.
\end{enumerate}
\end{prop}

\begin{proof} For the first point, proceed just as in the proof of Proposition 3.14 of \cite{GF}, using an induction on dimension, as well as the Kunneth isomorphism $H_*(X ; G) = H_*(X) \otimes_{\mathbb{Z}_2} H_*(G, \mathbb{Z}_2)$ (see remark \ref{remequivsinghom} (5)) when $X$ is compact and nonsingular. 

For the second point, use also an induction on dimension. For the compact case, proceed as in the proof of Proposition 3.15 of \cite{GF}, considering an equivariant resolution of the singularities of $X$. If $X$ is not compact, apply the previous case to $\overline{X}^{\mathcal{AS}}$ (the action of $G$ on $\overline{X}^{\mathcal{AS}}$ is free by definition) and the induction hypothesis to $\overline{X}^{\mathcal{AS}} \setminus X$ to obtain, by additivity of the equivariant virtual Poincar\'e series,
$$\beta(X ; G) = \beta\left(\overline{X}^{\mathcal{AS}} ; G\right) - \beta\left(\overline{X}^{\mathcal{AS}} \setminus X ; G\right) = \beta\left(\overline{X}^{\mathcal{AS}}/ G\right) - \beta\left((\overline{X}^{\mathcal{AS}} \setminus X)/G\right) = \beta(X/G).$$
\end{proof}

\begin{rem} As pointed out in example \ref{exequivpoincser} (2), our equivariant virtual Poincar\'e series does not encode the dimension, contrary to the equivariant virtual Poincar\'e series of \cite{GF} (Proposition 3.10). However, our equivariant virtual Poincar\'e series has been proven to be an invariant with respect to equivariant homeomorphism with $\mathcal{AS}$-graph, and we saw in example \ref{exequivpoincser} (2) that it could detect differences in some equivariant $\mathcal{AS}$-stuctures that the equivariant virtual Poincar\'e series of \cite{GF}, as well as the virtual Poincar\'e polynomial of the fixed points set, could not see. These three additive invariants should be thought as complementary.
\end{rem}

\end{document}